\documentclass[11pt,reqno]{amsart}
\usepackage[utf8]{inputenc}
\usepackage{amssymb}
\usepackage{amsmath}
\usepackage{amsthm}
\usepackage{amsfonts}
\usepackage{array}

\usepackage{bookmark}
\hypersetup{pdfstartview={FitH}}
\usepackage{hyperref}
\hypersetup{
     colorlinks=true,
     linkcolor=blue,
     citecolor=blue,
     filecolor=magenta,      
     urlcolor=cyan,
}

\theoremstyle{plain}
\newtheorem{theorem}{Theorem}[section]
\newtheorem{lemma}[theorem]{Lemma}

\newtheorem{proposition}[theorem]{Proposition}
\newtheorem{conjecture}{Conjecture}

\theoremstyle{definition}
\newtheorem{definition}[theorem]{Definition}

\newtheorem{rem}[theorem]{Remark} 
 
\numberwithin{equation}{section}
\newtheorem*{theorem*}{Theorem} 

\newcommand{\Z}{{\mathbb Z}}
\newcommand{\R}{{\mathbb R}}

\newcommand{\T}{{\mathcal T}}
\newcommand{\A}{{\mathcal A}}

\renewcommand{\T}{\mathcal{T}}

\DeclareMathOperator{\Gr}{\operatorname{Gr}}
\DeclareMathOperator{\lin}{\operatorname{span}}
\DeclareMathOperator{\dist}{\operatorname{dist}}

\DeclareFontFamily{U}{mathx}{\hyphenchar\font45}
\DeclareFontShape{U}{mathx}{m}{n}{
<5> <6> <7> <8> <9> <10>
<10.95> <12> <14.4> <17.28> <20.74> <24.88>
mathx10
}{}
\DeclareSymbolFont{mathx}{U}{mathx}{m}{n}
\DeclareFontSubstitution{U}{mathx}{m}{n}
\DeclareMathAccent{\widecheck}{0}{mathx}{"71}

\title[Trilinear singular Brascamp-Lieb integrals]{On trilinear singular Brascamp-Lieb integrals}

\author[Lars Becker]{Lars Becker}
\address{Mathematical Institute, University of Bonn, Endenicher Allee 60, 53115, Bonn, Germany}
\email{becker@math.uni-bonn.de}

\author{Polona Durcik}
\address{Schmid College of Science and Technology, Chapman University, 
One University Drive, Orange, CA 92866, USA}
\email{durcik@chapman.edu}

\author[Fred Yu-Hsiang Lin]{Fred Yu-Hsiang Lin}
\address{Mathematical Institute, University of Bonn, Endenicher Allee 60, 53115, Bonn, Germany}
\email{fredlin@math.uni-bonn.de} 

\date{October 31, 2024}


\begin{document}

\begin{abstract}
    We classify all trilinear singular Brascamp-Lieb forms, completing the classification in the two dimensional case by Demeter and Thiele \cite{Dem+2010}. We use known results in the representation theory of finite dimensional algebras, namely the classification of indecomposable representations of the four subspace quiver. 
    Our classification lays out a roadmap for achieving bounds for all degenerate higher dimensional bilinear Hilbert transforms. 
    As another step towards this goal, we prove new bounds for a particular class of forms that arises as a natural next candidate from our classification. We further prove conditional bounds for forms associated with mutually related  representations. For this purpose we develop a method of rotations that allows us to decompose any homogeneous $d$-dimensional singular integral kernel into $(d-1)$-dimensional kernels on hyperplanes.
\end{abstract} 
 
\maketitle

\section{Introduction}
This article continues the investigation of generalizations of the bilinear Hilbert transform 
\begin{equation}
    \label{eq BHT int}
    \operatorname{BHT}(f_1, f_2)(x) = \int_{\R} f_1(x+t) f_2(x+\alpha t) \frac{1}{t} dt\,,\qquad \alpha \ne 0,1\,,
\end{equation}
where the integral is understood as the principal value.
Lacey and Thiele in their breakthrough works \cite{lt1997,lacey1999calderon} proved boundedness of  $\operatorname{BHT}$ from $L^{p_1} \times L^{p_2}$ into $L^{p_0}$, provided that $\frac{1}{p_0} = \frac{1}{p_1} + \frac{1}{p_2} = 1$ and $p_0 > \frac{2}{3}$. This partially resolved a conjecture of Calderón \cite{calderonconj}. 

It is then very natural to ask about higher dimensional versions of \eqref{eq BHT int}, namely the operators
\begin{equation}
    \label{eq BHT int d}
    \operatorname{BHT}_d(f_1, f_2)(x) = \int_{\R^d} f_1(x+ A_1 t) f_2(x+A_2 t) K(t) dt\,,
\end{equation}
where $A_1, A_2: \R^d \to \R^d$ are linear maps and $K$ is a Calderón-Zygmund kernel on $\R^d$, defined below in \eqref{czkernel}. Demeter and Thiele \cite{Dem+2010} studied the two dimensional case $d = 2$ of \eqref{eq BHT int d}. The class of such operators is richer than in the one dimensional case, in that various levels of degeneracies occur depending on $A_1$ and $A_2$.  Demeter and Thiele found four qualitatively different cases, and prove boundedness for three of them using different tools. The final case was later resolved by Kovač \cite{twisted}, using again different techniques. 

In the present paper we extend this classification to the $d$-dimensional case and in fact to more general singular Brascamp-Lieb forms, in Theorem \ref{thm Hclass}. We require some definitions, which are set up in Sections \ref{subsec SBL} to \ref{subsec class SBL}.
We use our classification to fully characterize boundedness at exponents $p_1, p_2, p_3$ that do \emph{not} satisfy the Hölder relation $\frac{1}{p_1} + \frac{1}{p_2} + \frac{1}{p_3} = 1$, in Theorem \ref{thm i)ii)}.
In Section \ref{subsec proj} we further give three conditional bounds, Theorem \ref{thm summand}, Theorem \ref{thm n} and Theorem \ref{thm rot}. They indicate how the difficulty of algebraically related cases in the classification compares. We put our classification into context and discuss which cases are covered by the existing literature in Section \ref{subsec lit}. Finally, we give new bounds for a large class of cases with Hölder exponents in Theorem \ref{thm 3 twisted}. 

\subsection{Singular Brascamp-Lieb forms}
\label{subsec SBL}
By duality, bounds for the bilinear operators \eqref{eq BHT int d} are equivalent to bounds for the trilinear forms
\[
    \int_{\R^{2d}} f_1(x+A_1t) f_2(x+A_2t) f_3(x)  K(t) dt \, dx\,.
\]
Motivated by this, we make the following general definitions.

\begin{definition}
    An $l$-Calderón-Zygmund kernel is a tempered distribution $K$ on a Hilbert space $H$, such that $K$ agrees with a function away from $0$ and such that for any choice of orthonormal basis, the corresponding partial derivatives of the Fourier transform $\widehat K$ satisfy for all $\xi \ne 0$
    \begin{equation}
        \label{czkernel}
        |\partial^{\alpha}\widehat{K}(\xi)|\leq |\xi|^{-|\alpha|}\,, \quad |\alpha| \le l\,.
    \end{equation}
    Here the Fourier transform of a Schwartz function is defined by
    \[
        \widehat f(\xi) = \int e^{-2\pi i \xi \cdot x} f(x) \, dx\,,
    \]
    and this definition is extended to tempered distributions by density.
\end{definition}
\begin{definition}
    We define a (trilinear) \emph{singular Brascamp-Lieb datum} to be a tuple $\mathbf{H} = (H; H_0, H_1, H_2, H_3; \Pi_0, \Pi_1, \Pi_2, \Pi_3)$ of five finite dimensional Hilbert spaces $H,H_i$ and of four surjective linear maps $\Pi_i : H \to H_i$. 
\end{definition}
\begin{definition}
    Given a singular Brascamp-Lieb datum $\mathbf{H}$  and a Calderón-Zygumnd kernel $K$ on $H_0$, the associated \emph{singular Brascamp-Lieb form} $\Lambda_{\mathbf{H}}$ is the trilinear form defined a priori on Schwartz functions $f_i \in \mathcal{S}(H_i)$ by
    \begin{equation}
    \label{eq Brascamp Lieb}
    \Lambda_{\mathbf{H}}(K, f_1,f_2,f_3)=\int_{H} f_1(\Pi_{1}x)f_2(\Pi_{2}x)f_3(\Pi_{3}x)K(\Pi_{0}x) \, dx\,.
\end{equation}
\end{definition}

Our goal is to study Lebesgue space estimates 
\begin{equation}
    \label{eq holder}
    |\Lambda(K, f_1, f_2, f_3)| \leq C(l) \|f_1\|_{p_1}\|f_2\|_{p_2}\|f_3\|_{p_3}
\end{equation}
for singular Brascamp-Lieb forms and exponents $\mathbf{p} = (p_1, p_2, p_3)$. This motivates the following definition.

\begin{definition}
    We say that a form $\Lambda_{\mathbf{H}}$ and the datum $\mathbf{H}$ are \emph{$\mathbf{p}$-bounded} if there exists $l$ such that \eqref{eq holder} holds for all $f_1, f_2, f_3$ and all $l$-Calderón-Zygmund kernels $K$.
    We say that it is is of \emph{Hölder type} if it is $\mathbf{p}$-bounded for some $1 <  p_1, p_2, p_3 < \infty$ with $\frac{1}{p_1} + \frac{1}{p_2} + \frac{1}{p_3} = 1$. We will abbreviate $a < p_1, p_2, p_3 < b$ by $a < \mathbf{p} < b$.
\end{definition}

The methods used in previous literature to prove or disprove bounds \eqref{eq holder} vary substantially depending on  $\mathbf{H}$. This shows in the very different methods used in \cite{Dem+2010} and \cite{twisted}, and also in the analysis in \cite{Dem+2010} for different $\mathbf{H}$.
The following notion of equivalence is relevant for deciding boundedness of a singular Brascamp-Lieb form, as expressed in Lemma \ref{equivalent bounds}. 

\begin{definition}
    \label{def equiv}
    We call two singular Brascamp-Lieb data $\mathbf{H}$, $\mathbf{H}'$ equivalent if there exist invertible linear maps 
    \begin{equation}
        \label{eq morph 1}
        \varphi: H \to H'\,,\quad 
        \varphi_i: H_i \to H_i', \quad i = 0,1,2,3\,,
    \end{equation}
    such that 
    \begin{equation}
        \label{eq morph 2}
        \Pi_i' \circ \varphi = \varphi_{i} \circ \Pi_{i}\,, \quad i = 0,1,2,3\,.
    \end{equation}
\end{definition}

\begin{lemma}
    \label{equivalent bounds}
    Suppose that $\mathbf{H}$ and ${\mathbf{H}'}$ are equivalent singular Brascamp-Lieb data. Then for all $\mathbf{p}$, the form $\Lambda_\mathbf{H}$ is $\mathbf{p}$-bounded if and only  $\Lambda_{\mathbf{H}'}$ is $\mathbf{p}$-bounded.
\end{lemma}
Lemma \ref{equivalent bounds} is a direct consequence of changes of variables in the functions and the integral defining the singular Brascamp-Lieb form. 

Our goal is to classify $\mathbf{p}$-bounded singular Brascamp-Lieb forms up to equivalence. Note that the notions of $\mathbf{p}$-boundedness and equivalence of data are insensitive to the Hilbert space structures on the spaces in $\mathbf{H}, \mathbf{H}'$. Hence, only the underlying vector spaces and linear maps will be relevant for our classification. However, to make sense of  \eqref{czkernel} and \eqref{eq holder}, we need Lebesgue measures on the spaces $H, H_i$, and a norm on $H_0^*$. The $H, H_i$ are defined to be Hilbert spaces  to simplify the exposition, because Hilbert spaces canonically have this additional structure. (The same choice is made in \cite{BraLieb}, for similar reasons.) 

\begin{rem}
    \label{rem quant}
    To study quantitative estimates, that is, the size of the constant $C$ in \eqref{eq holder}, one needs a finer equivalence relation than the one given by \eqref{eq morph 1}, \eqref{eq morph 2}. Namely one should assume that $\varphi_0$ is a scalar multiple of an orthogonal transformation, $\varphi_0 \in \R \cdot O(H_0, H_0')$. This is because only scalar multiples of isometries preserve all quantitative assumptions on the Calderón-Zygmund kernels. Equivalence classes modulo this finer equivalence relation are parametrized by equivalence classes according to Definition \ref{def equiv} together with an element of $Gl(H_0') /(\R \cdot O(H_0, H_0'))$. The latter can be parametrized by nonsingular lower triangular matrices with a $1$ in the upper left corner.
\end{rem}

\subsection{The four subspace problem}
%\label{subsec 4sub}
The classification of Brascamp-Lieb data up to equivalence is equivalent to the so-called four subspace problem, which we now describe.

\begin{definition}
    A \emph{module} is a tuple $\mathbf{M} = (M; M_0, M_1, M_2, M_3)$ of a finite dimensional vector space $M$ and four subspaces $M_i \subseteq M$, $i=0,1,2,3$.
\end{definition}

Structures $\mathbf{M}$ are also called representations (of the four subspace quiver). We call them modules, because they are modules over the path algebra associated with that quiver. The interested reader is refered to \cite{Derksen+survey} for a short survey on quiver representations.

\begin{definition}
    Two modules $\mathbf{M}$ and $\mathbf{M'}$ are \emph{isomorphic} if there exists an invertible linear map $\psi: M \to M'$
    such that 
    \[
        \psi(M_i) = M_i'\,, \quad i = 0,1,2,3\,.
    \]
    If $\mathbf{M}$ is isomorphic to $\mathbf{M}'$, we write $\mathbf{M} \cong \mathbf{M}'$.
\end{definition}

The four subspace problem asks for a classification of all modules up to isomorphism. It was solved by Gelfand and Ponomarev \cite{Gel+1972} for algebraically closed fields. In the case of general fields (we are interested in $\R$), the solution was given by Nazarova \cite{Nazarova1, Nazarova2}. See also \cite{Med+2004} for an elementary proof. The solution consists of a list of indecomposable modules, such that each module is isomorphic to a unique (up to permutation) finite direct sum of indecomposables.

\begin{definition}
    The \emph{direct sum} of two modules $\mathbf{M} = (M; M_0, M_1, M_2, M_3)$ and $\mathbf{M}' = (M'; M_0', M_1', M_2', M_3')$ is defined to be the module
    \[
        \mathbf{M} \oplus \mathbf{M'} = (M \oplus M'; M_0 \oplus M_0', M_1 \oplus M_1', M_2 \oplus M_2', M_3 \oplus M_3')\,. 
    \]
\end{definition}

\begin{theorem}[Gelfand, Ponomarev \cite{Gel+1972}; Nazarova \cite{Nazarova1, Nazarova2}]
    \label{thm mod class}
    Let $\mathbf{M}$ be a module. Then there exists a finite sequence of modules $\mathbf{M}_1, \mathbf{M}_2, \dotsc, \mathbf{M}_k$, from the list in Table~\ref{table 4sub} (possibly after permuting the subspaces), such that 
    \[
        \mathbf{M} \cong \mathbf{M}_1 \oplus \dotsb \oplus \mathbf{M}_k\,.
    \]
    For every such representation 
    \[
        \mathbf{M} \cong \mathbf{M}_1'\oplus \dotsb \oplus \mathbf{M}_k'\,,
    \]
    there exists a permutation $\pi: \{1, \dotsc, k\} \to \{1, \dotsc, k\}$ such that $\mathbf{M}_i = \mathbf{M}'_{\pi(i)}$,  $i = 1, \dotsc, k$.
\end{theorem}

\begin{rem}
    The indecomposable modules $\mathbf{M}_j$ are listed in Table \ref{table 4sub} only up to permutation of the subspaces. We give the additional information which permutations give rise to isomorphic modules for each case in Lemma \ref{lem perm} in Appendix \ref{sec classification}. This information will be used in the proof of Theorem \ref{thm Hclass}. 
\end{rem}

We alert the reader that the same classification problem for more than four subspace, relevant to more-than-three linear forms, is significantly harder.
More precisely, the four subspace problem is the last in this sequence which is \emph{tame}. For tame classification problems, there exist for every dimension tuple of the involved subspaces only finitely many one parameter families of indecomposable modules, plus possibly finitely many additional indecomposable modules (at least if the underlying field is algebraically closed), see \cite{Crawley-Boevey}. If a classification problem is not tame, then it is \emph{wild}. One can show that every wild classification problem is at least as hard as the classification of finite dimensional modules up to isomorphism over \emph{any} finitely generated algebra. For both of these facts, and further references, see \cite{Drozd_tamewild}. In that sense wild classification problems are substantially more difficult.

Of course, we are not interested in all modules, but just in those corresponding to $\mathbf{p}$-bounded forms with $\mathbf{p} < \infty$. This imposes some restrictions, see Lemma \ref{lem necessary} and Lemma \ref{lem nec hol}. However, even with these additional restrictions, even if we additionally assume the Hölder condition $1/p_1 + 1/p_2 + 1/p_3 = 1$,  the classification problem remains wild for four- and higher linear forms, see Remark \ref{rem impossible}.

\subsection{Classification of singular Brascamp-Lieb forms}
\label{subsec class SBL}

Taking adjoint gives rise to a natural correspondence between the underlying vector spaces of singular Brascamp-Lieb data and modules. 

\begin{definition}
Let $\mathbf{H}$ be a singular Brascamp-Lieb datum.
The associated module $\mathbf{M}_{\mathbf{H}}$ is defined to be
\[
    \mathbf{M}_{\mathbf{H}} = (H^*; \Pi_0^* H_0^*, \Pi_1^* H_1^*, \Pi_2^* H_2^*, \Pi_3^* H_3^*)\,.
\]
Here $^*$ denotes adjoints and dual spaces. Conversely, if $\mathbf{M}$ is a module, then we associate to it a singular Brascamp-Lieb datum 
\[
    \mathbf{H}_{\mathbf{M}} = (M^*; M_0^*, M_1^*, M_2^*, M_3^*; \iota_0^*, \iota_1^*, \iota_2^*, \iota_3^*)\,.
\]
Here $\iota_j$ denotes the inclusion map $\iota_j : M_j \to M$, and we equip the finite dimensional vector spaces $M^*, M_j^*$ with any Hilbert space structure.
\end{definition}

If we define morphisms of Brascamp-Lieb data to be tuples of (not necessarily invertible) linear maps $\varphi, \varphi_i$ satisfying \eqref{eq morph 1} and \eqref{eq morph 2}, then the maps $\mathbf{H} \mapsto \mathbf{M}_\mathbf{H}$ and $\mathbf{M} \mapsto \mathbf{H}_\mathbf{M}$ become mutually inverse dualities of categories. As a consequence of this fact and Theorem \ref{thm mod class} we immediately obtain a classification of all singular Brascamp-Lieb data. 

\begin{theorem}
    \label{thm equivalence}
    Singular Brascamp-Lieb data $\mathbf{H}$ and $\mathbf{H'}$ are equivalent if and only if $\mathbf{M}_\mathbf{H}$ and $\mathbf{M}_{\mathbf{H'}}$ are isomorphic. For each module $\mathbf{M}$, there exists a singular Brascamp-Lieb datum $\mathbf{H}$ with $\mathbf{M}_\mathbf{H} \cong \mathbf{M}$. As a consequence, for every singular Brascamp-Lieb datum $\mathbf{H}$, there exists a finite list of modules $\mathbf{M}_1, \dotsc, \mathbf{M}_k$ from Table \ref{table 4sub} such that 
    \begin{equation}
        \label{dsum dec}
         \mathbf{M}_\mathbf{H} \cong \mathbf{M}_1 \oplus \dotsb \oplus \mathbf{M}_k\,.
    \end{equation}
    This list is unique, up to permutation. Conversely, for every finite list $\mathbf{M}_1,\dotsc, \mathbf{M}_k$ there exists a unique up to equivalence singular Brascamp-Lieb datum $\mathbf{H}$ such that \eqref{dsum dec} holds.
\end{theorem}

Most of the singular Brascamp-Lieb forms as in \eqref{dsum dec} are not $\mathbf{p}$-bounded for any $\mathbf{p} < \infty$. We exclude the case where some $p_i = \infty$ to avoid certain cases where the maps $\Pi_i$ are not surjective on the kernel of $\Pi_0$, which would complicate our analysis while offering little additional insight. We have the following classification of $\mathbf{p}$-bounded forms with $\mathbf{p} < \infty$, which will be proved in Section \ref{sec proof class}.

\begin{theorem}
    \label{thm Hclass}
    Let $1 \le \mathbf{p} < \infty$ and let $\mathbf{H}$ be a $\mathbf{p}$-bounded singular Brascamp-Lieb datum with $H_1, H_2, H_3 \ne\{0\}$. Then one of the following holds, with the notation from Appendix~\ref{sec classification}.
    \begin{enumerate}
        \item[i)] (Bilinear Hölder-type) There exists an assignment $\{i, j, k\} = \{1,2,3\}$ such that $\frac{1}{p_j} = \frac{1}{p_k} = 1 - \frac{1}{p_i}$ and $n_1, n_2, n_3, n_4 \ge 0$ such that
        \begin{equation}
            \label{eq bil hol}
            \mathbf{M}_\mathbf{H} \cong (\mathbf{P}^{(j)})^{\oplus n_1} \oplus (\mathbf{K}^{(j)})^{\oplus n_2} \oplus (\mathbf{P}^{(k)})^{\oplus n_3} \oplus (\mathbf{K}^{(k)})^{\oplus n_4}\,.
        \end{equation}
        \item[ii)] (Young-type) We have $\mathbf{p} = (p_1,p_2,p_3)$ with $\frac{1}{p_1} + \frac{1}{p_2} + \frac{1}{p_3} = 2$. If $p_1, p_2, p_3 \ne 1$ then there exist $n_1, n_2 \ge 0$ such that
        \begin{equation}
            \label{eq young1}
            \mathbf{M}_\mathbf{H} \cong \mathbf{Y}^{\oplus n_1} \oplus \mathbf{Z}^{\oplus n_2}\,.
        \end{equation}
        If there is some $i \in \{1,2,3\}$ with $p_i = 1$, then there exist $n_1, n_2, n_3, n_4 \ge 0$ such that
        \begin{equation}
            \label{eq young2}
            \mathbf{M}_\mathbf{H} \cong \mathbf{Y}^{\oplus n_1} \oplus \mathbf{Z}^{\oplus n_2} \oplus (\mathbf{P}^{(i)})^{\oplus n_3} \oplus (\mathbf{K}^{(i)})^{\oplus n_4}\,.
        \end{equation}
        \item[iii)] (Loomis-Whitney-type) We have $\mathbf{p} = (2,2,2)$ and there exist $n_1 , n_2 \ge 0$ and a list of modules $\mathbf{M}_1, \dotsc, \mathbf{M}_k$ from Table \ref{table Holder} with 
        \begin{equation}
            \label{eq lw}
            \mathbf{M}_\mathbf{H} \cong \mathbf{L}^{\oplus n_1} \oplus \mathbf{B}^{\oplus n_2} \oplus  \mathbf{M}_1 \oplus \dotsb \oplus \mathbf{M}_k\,.
        \end{equation}
        \item[iv)] (Hölder-type) We have $\mathbf{p} = (p_1, p_2, p_3)$ with $\frac{1}{p_1} + \frac{1}{p_2} + \frac{1}{p_3} = 1$. In this case, there exists a finite list of modules $\mathbf{M}_1, \dotsc, \mathbf{M}_k$ from Table \ref{table htype} such that 
        \[
             \mathbf{M}_\mathbf{H} \cong \mathbf{M}_1 \oplus \dotsb \oplus \mathbf{M}_k\,.
        \]
    \end{enumerate}
\end{theorem}

The proof of Theorem \ref{thm Hclass} uses necessary conditions for boundedness of nonsingular Brascamp-Lieb forms from \cite{BraLieb}. They can be applied to singular Brascamp-Lieb forms because the Dirac $\delta$ distribution is a Calderón-Zygmund kernel, and singular Brascamp-Lieb forms with kernel $\delta$ simplify to nonsingular Brascamp-Lieb forms. A similar argument previously appeared in \cite{SBLsurvey}. 

The singular Brascamp-Lieb forms corresponding to cases i), ii) of Theorem \ref{thm Hclass} are easily seen to be bounded by Hölder's inequality, Young's convolution inequality, and classical linear singular integral theory. The forms corresponding to case iii) are also bounded, by an elementary argument using Plancherel and the Loomis-Whitney inequality. This is summarized by the following theorem, which we prove in Section \ref{sec i)ii)iii)}.

\begin{theorem}
    \label{thm i)ii)}
    Let $\mathbf{M}_\mathbf{H}$ and $\mathbf{p}$ be as in case i), ii) or iii) of Theorem \ref{thm Hclass}. 
    Then $\mathbf{H}$ is $\mathbf{p}$-bounded. 
\end{theorem}

As Theorem \ref{thm i)ii)} shows, forms of Hölder type are the most interesting ones.  Showing boundedness for them is in general open, and contains some difficult problems.
We collect results from the literature, proving bounds in some cases, in Section \ref{subsec lit}.

\subsection{Projection results and method of rotations}
\label{subsec proj}

The difficulty of estimating singular Brascamp-Lieb forms increases when taking direct sums of the corresponding modules, in the following precise sense. 

\begin{theorem}
    \label{thm summand}
    Let $\mathbf{M}, \mathbf{M}'$ be two modules and let $\mathbf{p} < \infty$. Let $\mathbf{H}$ and $\mathbf{H} \oplus \mathbf{H}'$ be data with $\mathbf{M}_\mathbf{H} \cong \mathbf{M}$ and $\mathbf{M}_{\mathbf{H}\oplus \mathbf{H}'} \cong \mathbf{M}\oplus \mathbf{M}'$.
    Suppose that for each $l$-Calderón-Zygmund kernel $K$ we have 
    \[
        |\Lambda_{\mathbf{H} \oplus \mathbf{H}'}(K, f_1, f_2, f_3)| \leq C \|f_1\|_{p_1}\|f_2\|_{p_2} \|f_3\|_{p_3}\,.
    \]
    Then there exists a constant $C'$ such that for each $2l$-Calderón-Zygmund kernel $K$ we have
    \[
        |\Lambda_{\mathbf{H}}(K, f_1, f_2, f_3)| \leq C' \|f_1\|_{p_1}\|f_2\|_{p_2} \|f_3\|_{p_3}\,.
    \]
\end{theorem}

Also, in the case $\mathbf{J}_n^{(2)}$ of the classification, the difficulty increases with the parameter $n$. 

\begin{theorem}
    \label{thm n}
    Let $\mathbf{M}$ be a module and let $\mathbf{p} < \infty$. Let $\mathbf{H}_n$ and $\mathbf{H}_{n-1}$ be data with $\mathbf{M}_{\mathbf{H}_n} \cong \mathbf{M} \oplus \mathbf{J}^{(2)}_n$ and $\mathbf{M}_{\mathbf{H}_{n-1}} \cong \mathbf{M}\oplus \mathbf{J}^{(2)}_{n-1}$.
    Suppose that there exists $C$ such that for each $l$-Calderón-Zygmund kernel $K$ we have
    \[
        |\Lambda_{\mathbf{H}_n}(K, f_1, f_2, f_3)| \leq C \|f_1\|_{p_1}\|f_2\|_{p_2} \|f_3\|_{p_3}\,.
    \]
    Then there exists a constant $C'$ such that for each $2l$-Calderón-Zygmund kernel $K$ we have
    \[
        |\Lambda_{\mathbf{H}_{n-1}}(K, f_1, f_2, f_3)| \leq C' \|f_1\|_{p_1}\|f_2\|_{p_2} \|f_3\|_{p_3}\,.
    \]
\end{theorem}

The same is true for $\mathbf{J}_n^{(1)}$ and $\mathbf{J}_n^{(3)}$, because they are isomorphic to modules that can be obtained from $\mathbf{J}_n^{(2)}$ by permuting the subspaces, see Lemma \ref{lem perm}. We do not state similar theorems for $\mathbf{C}_n$ or $\mathbf{N}_n$, because Theorem \ref{thm nondegenerate} already gives unconditional bounds in these cases. For $\mathbf{T}_n$, we do not expect an analogue of Theorem \ref{thm n} to be true. The reason is that the associated forms become more singular as $n$ gets smaller, at least judging only by the number of arguments of the kernel compared to the functions.  

In a different direction, it is possible to express forms with a kernel taking $d$ arguments as superpositions of certain forms with kernels taking $d-1$ arguments. Thus bounds for the former are at most as hard as integrable bounds for the latter. A classical instance of this idea is the method of rotations, introduced by Calderón and Zygmund in \cite{CZrotation}, in which one expresses an odd Calderón-Zygmund kernel as a superposition of Hilbert transforms. Using this, one can deduce bounds for odd kernel Calderón-Zygmund operators in higher dimensions from the boundedness of the Hilbert transform. We prove a stronger version of this fact. Namely, \emph{every} Calderón-Zygmund kernel in dimension $3$ or higher can be expressed as a superposition of $2$-dimensional Calderón-Zygmund kernels on $2$-dimensional subspaces. 

This yields the following theorem for singular Brascamp-Lieb forms, which is proved in Section \ref{sec rot}. We denote by $\Gr_d(V)$ the Grassmann-manifold of $d$-dimensional subspaces of some vector space $V$. A Calderón-Zygmund kernel on a $d$-dimensional Hilbert space is called homogeneous if for all $x \ne 0$ it holds $K(tx) = t^{-d}K(x)$.

\begin{theorem}
    \label{thm rot}
    Let $\mathbf{H}$ be a singular Brascamp-Lieb datum and suppose that $d = \dim H_0 \ge 3$. Let $l \ge d+1$. There exists $C' > 0$ such that the following holds. For each $\theta \in \Gr_{d-1}(H_0)$, consider the datum
    \[
        \mathbf{H}(\theta) = \mathbf{H} \cap \Pi_0^{-1}(\theta) = (\Pi_0^{-1}(\theta), \theta, H_1, H_2, H_3; \Pi_0, \Pi_1, \Pi_2, \Pi_3)\,.
    \]
    Here we abuse notation and denote the restriction of $\Pi_i$ to $\Pi_0^{-1}(\theta)$ still by $\Pi_i$.
    Suppose that for all $\theta \in \Gr_{d-1}(M_0)$ there exists $C(\theta)$ such that for all homogeneous $l - \left\lceil\frac{d + 2}{2}\right\rceil$-Calderón-Zygmund kernels $K$ on $\theta$, we have
    \[
        |\Lambda_{\mathbf{H}(\theta)}(K, f_1, f_2, f_3)| \le C(\theta) \|f_1\|_{p_1} \|f_2\|_{p_2} \|f_3\|_{p_3}\,.
    \]
    Then for all homogeneous $l$-Calderón-Zygmund kernels $K$ on $H_0$, we have
    \[
        |\Lambda_{\mathbf{H}}(K, f_1, f_2, f_3)| \le C' \int C(\theta) \, d\theta \cdot \|f_1\|_{p_1} \|f_2\|_{p_2} \|f_3\|_{p_3}\,,
    \]
    where the integration is over $\Gr_{d-1}(H_0)$ with respect to the unique rotation invariant probability measure.
    If $d = 2$, then the same result is true under the additional assumption that $K$ is odd.
\end{theorem}

\begin{rem}
    The loss of derivatives is a  technical consequence of the fact that we assume Mikhlin bounds \eqref{czkernel} on the Fourier transform of the kernel. If the smoothness assumptions on the kernels are formulated on the spatial side, then there is no loss of derivatives.
\end{rem}

There is a crucial difficulty in applying Theorem \ref{thm rot}: It assumes quantitative, integrable estimates for the norms of $\Lambda_{\mathbf{H}(\theta)}$. Such estimates are notoriously hard to prove. See \cite{uraltsev2022uniform}, \cite{fraccaroli+2024} for the strongest currently known results in that direction, which still only apply to the case $\mathbf{N}_1$. Recall also that to study quantitative bounds, one should use the finer equivalence relation with $\varphi_0$ a scalar multiple of an orthogonal transformation, as described in Remark~\ref{rem quant}.

\subsection{Positive boundedness results in the literature}
\label{subsec lit}

It is tempting to conjecture that the conditions in Theorem \ref{thm Hclass} are already sufficient. By Theorem \ref{thm i)ii)}, this would follow from the following conjecture.

\begin{conjecture}
%    \label{con 1}
    All singular Brascamp-Lieb data $\mathbf{H}_\mathbf{M}$ with $\mathbf{M}$ as in case iv) of Theorem \ref{thm Hclass} are $\mathbf{p}$-bounded for all $1 < \mathbf{p} < \infty$ with $\frac{1}{p_1} + \frac{1}{p_2} + \frac{1}{p_3} = 1$.
\end{conjecture}

We now give a list of known boundedness results for forms of Hölder type. With the exception of Theorem \ref{thm 3 twisted}, these results are not new, however some of them have not been stated in this form anywhere in the literature. In what follows, we will always fix a basis and identify finite dimensional Hilbert spaces with $\R^n$, for some $n$.

Note first that the module $\mathbf{C}_0$ corresponds simply to Hölder's inequality in three functions. If a datum $\mathbf{H}_\mathbf{M}$ is $\mathbf{p}$-bounded then so is $\mathbf{H}_{\mathbf{M} \oplus \mathbf{C}_0}$, by Fubini and Hölder's inequality. Keeping this in mind, we can ignore $\mathbf{C}_0$ in the following discussion. 

\subsubsection{Coifman and Meyer}
The first result on multilinear singular integral operators, due to Coifman and Meyer \cite{coifman1978dela,coifman1978commutateurs}, treats the case $\mathbf{M} = \mathbf{C}_1^{\oplus n}$ for $n \ge 1$. The singular Brascamp-Lieb form corresponding to this module is 
\[
    \Lambda(K, f_1, f_2, f_3) = \int_{\R^n} \int_{\R^n} \int_{\R^n} f_1(x) f_2(x+y) f_3(x + z) K(y,z) \, dy \, dz \, dx\,.
\]
These forms are in a sense the least singular among all singular Brascamp-Lieb forms of Hölder type, because the kernel $K$ has the maximum possible number of arguments compared to the functions.

\subsubsection{Time-frequency analysis}
Lacey and Thiele \cite{lt1997,lacey1999calderon} proved bounds for the Bilinear Hilbert transform 
\[
    \Lambda(f_1, f_2, f_3) = \int_\R \int_\R f_1(x) f_2(x+t) f_3(x + \alpha t) \frac{1}{t} \, dt \, dx\,,
\]
where $\alpha \ne 0,1$. This corresponds to the module $\mathbf{N}_1$, with $X = \alpha$. Their methods were subsequently extended to treat also the cases $\mathbf{N}_n$ for $n \ge 2$, and direct sums thereof. For $n = 2$, this was done in \cite{Dem+2010}. For larger $n$, proofs can be found in \cite{RoosThesis, fraccaroli+2024}.

The techniques introduced by Lacey and Thiele apply to a certain class of multilinear Fourier multiplier operators more general than \eqref{eq Brascamp Lieb}. This was first observed in \cite{GilbertNahmod, muscalu2002multi}, where it is shown that it suffices if the multipliers satisfy symbol estimates away from some subspace, which in particular holds if they satisfy symbol estimates away from some smaller subspace. Using this observation, one can deduce also bounds for forms $\Lambda_\mathbf{H}$ with $\mathbf{M}_\mathbf{H}$ including summands $\mathbf{C}_n$. The following theorem summarizes this.

\begin{theorem}
    \label{thm type 03}
    Suppose that $\mathbf{M}$ is a direct sum of modules $\mathbf{N}_{n_i}$ and $\mathbf{C}_{m_i}$, for some finite sequences $n_i, m_i \in \mathbb{N}_{\ge 1}$. For each $2 < \mathbf{p} < \infty$ and each singular Brascamp-Lieb datum $\mathbf{H}$ associated with $\mathbf{M}$, there exists $C > 0$ and $l$ such that for each $l$-Calderón-Zygmund kernel $K$
    \[
        |\Lambda_{\mathbf{H}}(K, f_1,f_2, f_3)| \leq C \|f_1\|_{p_1}\|f_2\|_{p_2}\|f_3\|_{p_3}\,.
    \]
\end{theorem}

\subsubsection{One and a half dimensional time-frequency analysis}
The conditions of Theorem \ref{thm type 03} are generically satisfied. `Degenerate' cases were first studied by Demeter and Thiele in \cite{Dem+2010}, for functions of two arguments. There are, up to permutation of the functions and equivalence, only three degenerate cases when all functions and kernels have two arguments: $\mathbf{J}^{(i)}_2$, $\mathbf{N}_1 \oplus \mathbf{J}^{(i)}_1$ and $\mathbf{J}^{(i)}_1 \oplus \mathbf{J}^{(j)}_1$ for $i \ne j$. This can be read off of Table \ref{table Holder}, noting that changing $i$ and $j$ only amounts to permuting the functions. Demeter and Thiele develop a `one and a half-dimensional' time frequency analysis, to prove bounds for the cases $\mathbf{J}^{(i)}_2$ and $\mathbf{N}_1 \oplus \mathbf{J}^{(i)}_1$. 
Similarly as discussed before Theorem \ref{thm type 03}, their proof implies also boundedness of the less singular forms $\Lambda$ corresponding to $\mathbf{C}_1 \oplus \mathbf{J}^{(i)}_1$. Indeed, by performing a discretization of such $\Lambda$ as in \cite{Dem+2010}, one arrives at a model form that still specializes the form (3) in \cite[Section 3.1.1]{Dem+2010}, and is consequently bounded by the argument given there. 

Demeter and Thiele further observe that $\mathbf{p}$-bounds for the form $\Lambda_\mathbf{H}$, with $\mathbf{M}_\mathbf{H} = \mathbf{J}_2^{(i)}$, imply Carleson's theorem \cite{carleson66} on pointwise convergence of Fourier series of $L^{p_1}$ functions. By Theorems \ref{thm summand} and \ref{thm n}, the same is then true whenever $\mathbf{M}_\mathbf{H}$ has a direct summand $\mathbf{J}_n^{(i)}$ for any $n \ge 2$. 

\subsubsection{Twisted techniques}
Demeter and Thiele left open the last case in their classification, $\mathbf{J}^{(i)}_1 \oplus \mathbf{J}^{(j)}_1$ for $i \ne j$.
They called this case the `twisted-paraproduct'. It was later shown to be bounded by Kovač \cite{twisted}, using very different techniques.  
Variations of Kovač's techniques can  by applied to many other multilinear singular Brascamp-Lieb forms with so-called cubical structure, see \cite{DST22}. We expect that bounding forms associated with modules containing a direct summand other than $\mathbf{J}^{(i)}_1$ and $\mathbf{C}_1$ requires extensions of the time-frequency analysis methods described above, perhaps in combination with twisted techniques. For $\mathbf{N}_n$ this is suggested by the fact that all known proofs use such techniques, while for $\mathbf{J}_n^{(i)}$ the implication for Carleson's theorem offers some justification.
Thus, the only remaining trilinear singular Brascamp-Lieb forms that should be attackable using twisted techniques are the ones associated with $(\mathbf{J}^{(1)}_{1} \oplus \mathbf{J}^{(2)}_{1} \oplus \mathbf{J}^{(3)}_{1} \oplus \mathbf{C}_1)^{\oplus n}$, $n \ge 1$.

The following Theorem, which we will prove in Section \ref{sec twisted}, shows that they are indeed always $\mathbf{p}$-bounded.

\begin{theorem}
    \label{thm 3 twisted}
    Let $n \ge 1$ and let $\mathbf{M} = (\mathbf{J}^{(1)}_{1} \oplus \mathbf{J}^{(2)}_{1} \oplus \mathbf{J}^{(3)}_{1} \oplus \mathbf{C}_1)^{\oplus n}$. Let $2 < \mathbf{p} < \infty$ and let $\mathbf{H}$ be a singular Brascamp-Lieb datum associated with $\mathbf{M}$. Then there exists $l$ and $C > 0$ such that for all $l$-Calderón-Zygmund kernels $K$, we have 
    \[
       |\Lambda_{\mathbf{H}}(K, f_1,f_2, f_3) |\leq C \|f_1\|_{p_1} \|f_2\|_{p_2} \|f_3\|_{p_2}\,.
    \]
\end{theorem}

Note that Theorem \ref{thm 3 twisted} recovers via Theorem \ref{thm summand} boundedness of the twisted paraproduct, as well as of certain higher dimensional versions. It further gives a new proof of boundedness of the form associated with $\mathbf{J}_1^{(j)} \oplus \mathbf{C}_1$, different from the one implicit in \cite{Dem+2010}.

 By a cone decomposition, the proof of Theorem \ref{thm 3 twisted} reduces to two essentially different cases. 
 The first case  can be treated using bounds for the standard maximal and square functions, in analogy with  the Coifman-Meyer multipliers. In this case we have, in fact, boundedness in a larger range $1< \mathbf{p} <\infty$. The second case is bounded using   twisted techniques, tailored to the specific structure of the form. The arguments rely on intertwined applications of   the Cauchy-Schwarz inequality, integration-by-parts identity, and positivity arguments. These arguments are applied in a localized setting, which in turn gives the claimed range of boundedness. It is an open problem to further lower the range of exponents in Theorem \ref{thm 3 twisted}. For the  twisted paraproduct \cite{twisted}, fiber-wise Calder{\'o}n-Zygmund decomposition \cite{Ber12} can be used to extend the range of some exponents. However, similar arguments do not seem to directly apply in our setting.

\subsubsection{The triangular Hilbert transform}
The final family of the classification, $\mathbf{T}_n$, contains and generalizes the so-called triangular Hilbert transform
\[
    \Lambda(f_1, f_2, f_3) = \int_\R \int_\R \int_\R f_1(x_1, x_2) f_2(x_1 + t, x_2) f_3(x_1, x_2 + t) \frac{1}{t} \, dt \, dx_1 \, dx_2\,.
\]
Proving any $\mathbf{p}$-bounds for this form is a hard open problem. We refer to \cite{specialfunction} for some discussion and a partial result. By Theorem \ref{thm summand}, bounding any form associated with a module $\mathbf{M}$ with a direct summand $\mathbf{T}_n$ is at least as hard as bounding $\mathbf{T}_n$, and therefore also open.
This applies in particular, but not exclusively, to all forms where $\dim H_0 < \dim H_1$ (note that we have $\dim H_1 = \dim H_2 = \dim H_3$ for every $\mathbf{H}$ of Hölder type). Indeed, all such forms must contain an indecomposable direct summand satisfying the same inequality, and the only such summands are of type $\mathbf{T}_n$. We note that this invalidates a certain claim in the paper \cite{demeter2010multilinear}, in the case of trilinear forms. More precisely, in \cite{demeter2010multilinear} boundedness of certain multilinear singular Brascamp-Lieb forms is shown, which do have a dimension deficit as above, and it is claimed that the assumptions placed on the forms are generically satisfied. We believe that this genericity claim is false for trilinear forms.

\subsubsection{Further questions}
At the time of writing, the above list of cases where $\mathbf{p}$-bounds are known is complete, to the best of our knowledge. This gives rise to a number of open questions. We consider it an interesting question whether time frequency techniques and twisted techniques can be combined, natural test cases are $\mathbf{N}_1 \oplus \mathbf{J}^{(1)}_{1} \oplus \mathbf{J}^{(2)}_{1}$ or $\mathbf{N}_1 \oplus \mathbf{J}^{(1)}_{1} \oplus \mathbf{J}^{(2)}_{1} \oplus \mathbf{J}^{(3)}_{1}$. More difficult seems to be the question of bounds for the forms associated with $\mathbf{T}_n$, $n \ge 2$. Judging by the dimension of the space $H_0$ in relation to $H_1, H_2, H_3$ alone, these forms become less singular as $n$ increases, so these might be useful test cases towards the triangular Hilbert transform. Finally, it would be interesting to gain a better understanding of the questions considered in this paper for higher degrees of multilinearity. While a classification in terms of direct summands is not possible, see Remark \ref{rem impossible}, there might be a different algebraic description of the properties of modules relevant for proving $\mathbf{p}$-bounds.

\subsection{Comparison with the literature}
We point out that various related objects have been studied under the name singular Brascamp-Lieb forms. Some completely nondegenerate cases with higher degrees of multilinearity have been considered in \cite{muscalu2002multi, demeter2010multilinear}, using time frequency analysis. Some further multilinear cases with so-called `cubical structure' are studied in \cite{DST22, durcik2023norm}, using twisted techniques. However, there has been no attempt of a systematic study of all degenerate cases. 

Our kernel $K$ always satisfies the single parameter Mikhlin condition \eqref{czkernel}. This is in contrast to related multiparameter problems, which have been studied for example in \cite{benea2016multiple,muscalu2004bi,muscalu2006multi,muscalu2022five} in connection with fractional Leibniz-rules. In particular, we point out that the `tensorization' of forms to obtain multiparameter forms, as for example in \cite{benea2016multiple}, is not the same as the procedure of taking direct sums of modules associated with singular Brascamp-Lieb forms. The former is a way of constructing multiparameter forms, while the latter constructs single parameter forms. However, as the Mikhlin condition \eqref{czkernel} on $\R^n$ is implied by an $n$-parameter kernel condition, known multiparameter bounds imply some of the one parameter bounds. In particular, we note that the bounds obtained in \cite[Theorem 6]{benea2016multiple} for a tensor product of one bilinear Hilbert transform with $n$ many paraproducts imply boundedness of the forms of type $\mathbf{N}_1 \oplus \mathbf{C}_1^n$.

Another question concerning the kernels is about the optimal regularity $l$, or more generally for optimal regularity conditions on the kernel $K$ or the symbol. Classical results for linear singular integral operators giving such sharp regularity conditions have been generalized to Coiffman-Meyer type forms $\mathbf{C}_1^{\oplus n}$ in
\cite{tomita2010hormander, grafakos2012hormander, muscalu2014calderon, lee2021hormander}, and to the bilinear Hilbert transform in \cite{chen2024}.

\subsection*{Acknowledgment} 
We thank Christoph Thiele for numerous discussions and his support in facilitating collaboration among the authors.
LB is supported by the Collaborative Research Center 1060 funded by the Deutsche
Forschungsgemeinschaft (DFG, German Research Foundation) and the Hausdorff Center for
Mathematics, funded by the DFG under Germany’s Excellence Strategy - EXC-2047/1 - 390685813. PD is   supported by the NSF grant  DMS-2154356.
FL is supported by the DAAD Graduate School Scholarship Programme - 57572629.
This work was completed while the  authors were in residence at the Hausdorff Research Institute for Mathematics in Bonn, during the trimester program ``Boolean Analysis in Computer Science", funded as well by the  DFG  under Germany's Excellence Strategy - EXC-2047/1 - 390685813.  

\section{The classification: Proof of Theorem \ref{thm equivalence} and Theorem \ref{thm  Hclass}}

\label{sec proof class}

\begin{proof}[Proof of Theorem \ref{thm equivalence}]
    The first claim of Theorem \ref{thm equivalence} follows immediately from basic linear algebra. The remaining claims follow from the first, the classification of modules in Theorem \ref{thm mod class} and the facts that clearly $\mathbf{M}_{\mathbf{H}_\mathbf{M}} \cong \mathbf{M}$ and $\mathbf{H}_{\mathbf{M}_\mathbf{H}} \cong \mathbf{H}$.
\end{proof}

Before proving Theorem \ref{thm Hclass}, we recall some necessary conditions for a datum $\mathbf{H}$ to be $\mathbf{p}$-bounded for some $\mathbf{p} < \infty$.
They were proven in \cite{SBLsurvey}, by adapting similar arguments for non-singular Brascamp-Lieb inequalities from \cite{BraLieb}.

\begin{lemma}
    \label{lem necessary}
    Let $\mathbf{p} < \infty$ and suppose that $\mathbf{H}$ is $\mathbf{p}$-bounded. Then for each $i = 1,2,3$,  
    \begin{equation}
        \label{eq SBL nec 1}
        \Pi_i \ker \Pi_0 = \Pi_i H\,.
    \end{equation}
    Furthermore, for each subspace $H' \subseteq \ker \Pi_0$, it holds that 
    \begin{equation}
    \label{eq SBL nec 2}
        \dim H' \le \sum_{i = 1}^3 \frac{\dim \Pi_i H'}{p_i}\,,
    \end{equation}
    and if $H' = \ker \Pi_0$, then we have equality in \eqref{eq SBL nec 2}.
\end{lemma}

\begin{proof}
    Note that the Dirac $\delta$ distribution is a Calderón-Zygmund kernel. Hence, if $\mathbf{H}$ is $\mathbf{p}$-bounded, then the Brascamp-Lieb form 
    \[
        \int_{H} f_1(\Pi_1(x)) f_2(\Pi_2(x)) f_3(\Pi_3(x)) \delta(\Pi_0(x)) \, dx = c \int_{\ker \Pi_0} f_1(\Pi_1(x)) f_2(\Pi_2(x)) f_3(\Pi_3(x)) \, dx
    \]
    is bounded on $L^{p_1} \times L^{p_2} \times L^{p_3}$. Theorem 1.13 in \cite{BraLieb} then immediately gives \eqref{eq SBL nec 2}. The first condition \eqref{eq SBL nec 1} follows from the fact that if $p_j < \infty$, then the Brascamp-Lieb form can only be bounded on $L^{p_j}$ if $\Pi_i|_{\ker \Pi_0}$ is surjective.
\end{proof}

\begin{lemma}
    \label{lem nec hol}
    Suppose that the datum $\mathbf{H}$ is of Hölder type. Then for each $i = 1,2,3$, we have that $H = \ker \Pi_0 \oplus \ker \Pi_i$.
\end{lemma}

\begin{proof}
    By \eqref{eq SBL nec 2} we have 
    \[
        \dim \ker \Pi_0 = \sum_{i = 1}^3 \frac{\dim \Pi_i \ker \Pi_0}{p_i} \le \sum_{i = 1}^3 \frac{\dim \ker \Pi_0}{p_i} = \dim \ker \Pi_0\,.
    \]
    In the last step we used that $\frac{1}{p_1} + \frac{1}{p_2} + \frac{1}{p_3} = 1$. So we must have equality in the middle, hence $\Pi_{j}\vert_{\operatorname{ker}\Pi_{0}}$ is injective, which gives that $\ker \Pi_j \cap \ker \Pi_0 = \{0\}$. Combining this with \eqref{eq SBL nec 1}, we obtain
    \[
        \dim \ker \Pi_0 = \dim \Pi_i \ker \Pi_0 = \dim \Pi_i H = \dim H - \dim \ker \Pi_i\,,
    \]
    which completes the proof of the lemma.
\end{proof}

\begin{proof}[Proof of Theorem \ref{thm Hclass}]
    Let $\mathbf{p} < \infty$ and suppose that $\mathbf{H}$ is $\mathbf{p}$-bounded. To simplify some formulas, we will write below $q_i = p_i^{-1}$. By Theorem \ref{thm mod class}, we have that
    \[
        \mathbf{M}_\mathbf{H} = \mathbf{M}_1 \oplus \dotsb \oplus \mathbf{M}_k
    \]
    for some modules $\mathbf{M}_k$ from Table \ref{table 4sub}. By Theorem \ref{thm summand}, each $\mathbf{H}_j = \mathbf{H}_{\mathbf{M}_j}$, $j=1,\dotsc, k$, is $\mathbf{p}$-bounded. 

    Recall that Theorem \ref{thm mod class} allows for permutations of the subspaces in the modules in Table \ref{table 4sub}, see Lemma \ref{lem perm} for an exact description of which permutations of the subspace give rise to nonisomorphic modules. We will denote modules from Table \ref{table 4sub} by adding the permutation and the parameter $n$ as subscripts. 

    We write 
    \[
        \mathbf{H}_j = (H_{j}, H_{j0}, H_{j1}, H_{j2}, H_{j3}, \Pi_{j0}, \Pi_{j1}, \Pi_{j2}, \Pi_{j3})\,.
    \]
    By condition \eqref{eq SBL nec 1} of Lemma \ref{lem necessary} and surjectivity of the maps $\Pi_{ji}$, we have for $i=1,2,3$
    \begin{equation}
        \label{eq dim cond}
        \dim H_{ji} = \dim \Pi_{ji} H_j = \dim \Pi_{ji} \ker \Pi_{j0} \le \dim H_j - \dim H_{j0}\,.
    \end{equation}
    By comparing with Table \ref{table 4sub}, this immediately implies that $\mathbf{M}_j \not\cong \mathbf{IV}_{n, \pi}^*$ and $\mathbf{M}_j \not\cong \mathbf{V}_{n,\pi}^*$ for any $n$ or $\pi$. 

    Suppose next that $\mathbf{M}_j \cong \mathbf{I}_{n, \pi}$ for some $n \ge 1$ and some permutation $\pi$. If $H_{j0}$ corresponds to the second subspace in the block matrix in Table \ref{table 4sub}, then the map $\Pi_{j3}$ corresponding to the fourth subspace is not surjective on $\ker \Pi_0$. Similarly, if $H_{j0}$ corresponds to the fourth subspace, then the map corresponding to the second one is not surjective on $\ker \Pi_0$. Thus by \eqref{eq SBL nec 1} $H_{j0}$ must correspond to the first or third subspace, and by Lemma \ref{lem perm} we can assume that it corresponds to the first. By \eqref{eq SBL nec 2}, $\mathbf{p}$ must satisfy the Hölder condition $q_1 + q_2 + q_3 = 1$. By permuting the last three subspaces and using Lemma \ref{lem perm}, we obtain three isomorphism classes of modules, which are exactly $\mathbf{J}_n^{(1)}, \mathbf{J}_n^{(2)}, \mathbf{J}_n^{(3)}$ in Table \ref{table htype}.

    Suppose now that $\mathbf{M}_j \cong \mathbf{I\!I}_{n,\pi}$ for some $n \ge 0$ and some permutation $\pi$. By \eqref{eq dim cond}, the permutation $\pi$ must be so that $H_{j0}$ corresponds to a subspace of dimension $n$.
    We will assume by permuting the functions that $H_{j3}$ is the other subspace of dimension $n$. Applying \eqref{eq SBL nec 2} to the full space $H' = \ker \Pi_{j0}$ we obtain 
    \[
        (n+1)q_1 + (n+1)q_2 + nq_3 = n+1\,.
    \]
    On the other hand, applying \eqref{eq SBL nec 2} to the one dimensional space $H' = \ker \Pi_{j3} \cap \ker \Pi_{j0}$ yields
    \[
        q_1 + q_2 \ge 1\,.
    \]
    Since $p_3 < \infty$ and hence $q_3 > 0$, it follows that $n = 0$. So in this case, we must have $\mathbf{M}_j \cong \mathbf{P}^{(3)}$ and $\frac{1}{p_1} + \frac{1}{p_2} = 1$. Note that swapping the two nonzero subspaces in $\mathbf{P}^{(j)}$ gives an isomorphic module, see Lemma \ref{lem perm}. Permuting the functions yields the two additional possibilities $\mathbf{M}_j \cong \mathbf{P}^{(1)}$ or $\mathbf{M}_j \cong \mathbf{P}^{(2)}$, with the corresponding conditions on $\mathbf{p}$. 

    Next, assume that $\mathbf{M}_j \cong \mathbf{IV}_{n,\pi}$ for some $n \ge 0$ and $\pi$. Suppose first that $H_{j0}$ corresponds to one of the first three subspaces in Table \ref{table 4sub}. We permute the subspaces so that $\dim H_{j3} = n$.
    Then we get from \eqref{eq SBL nec 2} that 
    \[
        (n+1)q_1 + (n+1)q_2 + nq_3 = n+1\,.
    \]
    Taking $H' = \ker \Pi_{j0} \cap \ker \Pi_{j3}$, we also have
    \[
        q_1 + q_2 \ge 1\,.
    \]
    Since $p_3 < \infty$ it follows that $n = 0$ and hence $\mathbf{M}_j \cong \mathbf{K}^{(3)}$ and $\frac{1}{p_1} + \frac{1}{p_2}  =1$. Note that swapping the two nonzero subspaces in $\mathbf{P}^{(j)}$ gives an isomorphic module. Thus, permuting the subspaces, only yields the additional possibilities $\mathbf{M}_j \cong \mathbf{K}^{(1)}$ or $\mathbf{M}_j \cong \mathbf{K}^{(2)}$, with corresponding conditions on $\mathbf{p}$.
    It remains to consider the case where $H_{j0}$ is the last subspace in Table \ref{table 4sub}. \eqref{eq SBL nec 2} applied to $H' = \ker \Pi_{j0}$ gives in this case
    \[
        (n+1)(q_1 + q_2 + q_2) = n+2\,.
    \]
    On the other hand, applying \eqref{eq SBL nec 2}
    to each of the one dimensional subspaces $\ker \Pi_{j0} \cap \ker \Pi_{ji}$, $i = 1,2,3$, and adding the resulting inequalities, yields
    \[
        2(q_1 + q_2 + q_3) \ge 3\,.
    \]
    Hence $n = 0$ and $\frac{1}{p_1} + \frac{1}{p_2} + \frac{1}{p_3} = 2$ or $n = 1$ and $p_1 = p_2 = p_3= 2$. This corresponds to $\mathbf{M}_j \cong \mathbf{Y}$ and $\mathbf{M}_j \cong \mathbf{L}$, respectively. Note that all permutations of the subspaces in these modules corresponding to functions yield isomorphic modules.

    Finally assume that $\mathbf{M}_j \cong \mathbf{V}_{n,\pi}$ for some $n \ge 0$ and $\pi$. Note that all such modules for fixed $n$ and  different $\pi$ are isomorphic. Applying condition \eqref{eq SBL nec 2} to $H' = \ker\Pi_{j0}$, we obtain
    \begin{equation*}
        n(q_1 + q_2 + q_3) = n+1\,.
    \end{equation*}
    On the other hand, applying \eqref{eq SBL nec 2} to $H' = \ker \Pi_{j0} \cap \ker \Pi_{ji}$, $i=1,2,3$, and adding the resulting inequalities, yields
    \begin{equation*}
        2(q_1 + q_2 + q_3) \ge 3\,.
    \end{equation*}
    Hence, we must have either
    $n = 1$ and $\frac{1}{p_1} + \frac{1}{p_2} + \frac{1}{p_3} = 2$ or $n = 2$ and $p_1 = p_2 = p_3 = 2$.
    This corresponds to $\mathbf{M}_j \cong \mathbf{Z}$ and $\mathbf{M}_j \cong \mathbf{B}$, respectively.

    In the remaining cases, it follows immediately from \eqref{eq dim cond} that $\mathbf{H}_{j0}$ must correspond to the first subspace. All permutations respecting this give rise to isomorphic modules, or for modules of type $\mathbf{0}$ to another module of type $\mathbf{0}$.
    From \eqref{eq SBL nec 2} it then follows that $\frac{1}{p_1} + \frac{1}{p_2} + \frac{1}{p_3} = 1$. 
    Case $\mathbf{0}$ then corresponds to $\mathbf{N}$, case $\mathbf{I}$ gives after permuting the subspaces rise to the Jordan block cases $\mathbf{J}^{(s)}$, case $\mathbf{I\!I\!I}$ corresponds to $\mathbf{C}$ and case $\mathbf{I\!I\!I}^*$ to $\mathbf{T}$.

    Thus the possible choices of $\mathbf{p}$ are exactly as in case i) - iv) of Theorem \ref{thm Hclass}. Collecting the possible summands for each choice of $\mathbf{p}$ completes the proof.
\end{proof}

\begin{rem}
    \label{rem impossible}
    We now show that already the classification of $n$-linear singular Brascamp-Lieb forms of Hölder type is as hard as the classification of representations of the $n-1$-Kronecker quiver, i.e. of tuples of $n-1$ linear maps between two finite dimensional vector spaces, up to isomorphism. This classification problem is wild for $n > 3$, see for example Theorem 1 and 2 in \cite{Nazarova2}. Thus a classification as above is not possible for any $n > 3$, not even under the assumption that the forms are of Hölder type.
    
    Note that the necessary conditions from both Lemma \ref{lem necessary} and Lemma \ref{lem nec hol} continue to hold for more than three functions, with identical proofs.  
    Suppose the $\mathbf{H}$ is a datum of Hölder type.
    Let $a = \dim H_0$ and $b = \dim H_1$. By Lemma \ref{lem nec hol}, we have $\ker \Pi_0 \oplus \ker \Pi_1 = H$, so we can choose bases of $H$, $H_0$ and $H_1$  such that the matrices of $\Pi_0, \Pi_1$ are given by
    \[
        \begin{pmatrix}
            I_a \\ 0
        \end{pmatrix} \quad \text{and} \quad
        \begin{pmatrix}
            0 \\ I_b
        \end{pmatrix}\,.
    \]
    Choosing in $H_i$, $2 \le i \le n$, the basis $\Pi_i(e_{a+1}), \dotsc \Pi_i(e_{a+b})$, the matrices of $\Pi_i, 2 \le i \le n$ are given by
    \[
        \begin{pmatrix}
            A_i \\ I_b
        \end{pmatrix}\,.
    \]
    for certain $a \times b$ matrices $A_i$. Let two singular Brascamp Lieb data $\mathbf{H}, \mathbf{H}'$ be given, and assume that they can be transformed into the above normal form with matrices $A_i$ and $A_i'$, $2 \le i \le n$, respectively. Then $\mathbf{H}, \mathbf{H}'$ are equivalent if and only if there exists an invertible $a \times a$ matrix $P$ and an invertible $b \times b$ matrix $Q$ such that for all $i = 2, \dotsc, n$
    \begin{equation}
    \label{eq kronecker equivalence}
        QA_iP = A_i'\,.
    \end{equation}

    On the other hand, the datum $(A_2, \dotsc, A_n)$ determines $n-1$ linear maps from $\R^b \to \R^a$, so a representation of the $(n-1)$-Kronecker quiver. Two such representations are also isomorphic if and only if \eqref{eq kronecker equivalence} holds. Thus classifying $n$-linear singular Brascamp-Lieb forms of Hölder type is as hard as classifying representations of the $n-1$-Kronecker quiver.
\end{rem}

\section{Bounds for forms of non-H\"older type: Proof of Theorem \ref{thm i)ii)}}
\label{sec i)ii)iii)}

We go through the cases one by one. For simplicity we omit  the domain of integration from the notation. Here and in the following sections, we will find constants for various related
inequalities, and by abuse of notation we will denote each of them by the letter $C$. In particular, the meaning of $C$ may change from line to line.

\subsection{Case i} 
Suppose first that $\mathbf{H}$ is as in the case i), that is, \eqref{eq bil hol} holds for $i,j,k\in \{1,2,3\}$.  We choose coordinates $x_1 \in \R^{n_1}$,  $x_2, u \in \R^{n_2}$, $y_1 \in \R^{n_3}$, $y_2, v \in \R^{n_4}$.   
To prove bounds for
$\Lambda_{\mathbf{H}}(f_{1},f_{2},f_{3})$, it suffices after a change of variables to prove bounds for  
\begin{equation*}
    \int f_{j}(x_1,x_2)f_{k}(y_1,y_2)f_{i}(x_1,y_1,x_2+u,y_2+v)K(u,v)\,dx_1\,dy_1\,dx_2 \,dy_2\,du\,dv
\end{equation*}
\[ = \int f_{j}(x_1,x_2)f_{k}(y_1,y_2)(f_{i}*\tilde{K})(x_1,y_1,x_2,y_2)\,dx_1\,dy_1\,dx_2 \,dy_2\,, \]
where $\tilde{K}(u,v) = K(-u,-v)$ and the convolution is in the third and fourth argument only.
Since $p_j < \infty$, we have $p_i>1$. Thus we can further estimate, using Hölder's inequality for the exponents $p_j,p_i$ with    $\frac{1}{p_{j}}=1-\frac{1}{p_{i}}$  and a linear singular integral bound on $f_i$
\begin{equation*}%\label{3singbdd}
    \le C \int \left\|f_{j}(x_1,x_2)f_{k}(y_1,y_2)\right\|_{L^{p_{j}}_{x_2,y_2}} \|f_{i}(x_1,y_1,x_2,y_2)\|_{L^{p_{i}}_{x_2,y_2}} \,dx_1\,dy_1\,.
\end{equation*}
By H\"older's inequality and the condition $p_{j}=p_{k}$, this is bounded by
\begin{equation*}%\label{3hold}
    \leq C\|f_{1}\|_{p_{1}}\|f_{2}\|_{p_{2}} \|f_{3}\|_{{p_{3}}}\,,
\end{equation*}
which completes the proof. If $n_2 = n_4 = 0$, then the assumption $p_j < \infty$ is not needed, and the estimate follows just from  Hölder's inequality.

\subsection{Case ii}
Suppose next that $\mathbf{H}$ is in the case ii) with $p_{1},p_{2},p_{3}\neq 1$, so \eqref{eq young1} holds. Again we choose coordinates $x_1, y_1 \in \R^{n_1}$, $x_2, y_2, z_2 \in \R^{n_2}$ and write the form $\Lambda_{\mathbf{H}}(f_{1},f_{2},f_{3})$ after a change of variables up to a constant as 
\begin{equation*}
         \int f_{1}(x_1+y_1,x_2+y_2)f_{2}(x_1,x_2+z_2)f_{3}(y_1,y_2)K(z_2)\,dx_1\,dx_2\,dy_1\,dy_2\,dz_2\,.   
\end{equation*}
Using $*$ to denote convolution in the second argument only, we estimate this with Young's convolution inequality, and then a linear singular integral bound using that $p_3 > 1$, by
\begin{equation*}
    \leq \|f_{1}\|_{p_{1}}\|f_{2}\ast\tilde{K}\|_{p_{2}}\|f_{3}\|_{p_{3}} \leq C\|f_{1}\|_{p_{1}}\|f_{2}\|_{p_{2}}\|f_{3}\|_{p_{3}}\,.
\end{equation*}

Now suppose that there is some $i$ with $p_{i}=1$, so \eqref{eq young2} holds. We may assume  $i=1$, because the conditions on $\mathbf{p}$ are otherwise symmetric.  We choose coordinates $x_3 \in \R^{n_3}$, $x_4, z_4 \in \R^{n_4}$ and write the form $\Lambda_{\mathbf{H}}(f_{1},f_{2},f_{3})$ up to a constant and a change of variables as
\begin{equation*}
    \int f_{1}(x_{1}+y_{1},x_{2}+y_{2})f_{2}(x_{1},x_{2},x_{3},x_{4})f_{3}(y_{1},y_{2}+z_2,x_{3},x_{4}+z_4)K(z_2,z_4)\,dx\,dy\,dz\,.
\end{equation*}
We recognize a convolution in the second and fourth coordinate of $f_3$ with $\tilde{K}$. 
Applying Hölder's inequality in  $x_3, x_4$, using that $\frac{1}{p_2} + \frac{1}{p_3} = 1$, we  bound the last display by
\begin{equation*}
    \int |f_{1}(x_{1}+y_{1},x_{2}+y_{2})|\left\|f_{2}(x_{1},x_{2},x_{3},x_{4})\right\|_{L^{p_{2}}_{x_{3},x_{4}}}\|f_{3}\, \ast_{2,4}\, \tilde{K}(y_{1},y_{2},x_{3},x_{4})    \|_{L^{p_{3}}_{x_{3},x_{4}}}dx_{1}dx_{2}dy_{1}dy_{2}\,.
\end{equation*}
By Young's convolution inequality and then a linear singular integral bound, using that $p_2 < \infty$ and hence $p_3 > 1$, this is again bounded by $C\|f_{1}\|_{p_1}\|f_{2}\|_{{p_{2}}} \|f_{3}\|_{{p_{3}}}$. 

\subsection{Case iii}
Suppose finally that $\mathbf{H}$ is as in case iii), so \eqref{eq lw} holds. We use Fourier inversion to express $\Lambda_\mathbf{H}(f_1, f_2, f_3)$ in terms of the Fourier transforms $\widehat f_1$, $\widehat f_2$ and $\widehat f_3$. Then we apply the triangle inequality to move absolute values inside and estimate  $\widehat{K}$  by $1$ using \eqref{czkernel}. The resulting expression is a non-singular Brascamp-Lieb form $\widehat \Lambda$ in $\widehat f_1$, $\widehat f_2$ and $\widehat f_3$.
By Plancherel's theorem, the problem thus reduces to checking that this Brascamp-Lieb form is bounded at exponent $(p_{1},p_{2},p_{3})=(2,2,2)$. 
Transferring a Brascamp-Lieb form to the Fourier side in this way commutes with taking direct sums of the associated modules. 
By Lemma 4.8 in \cite{BraLieb}, a Brascamp-Lieb form is $\mathbf{p}$-bounded if each direct summand is $\mathbf{p}$-bounded. Thus it suffices to verify $(2,2,2)$-boundedness of each possible direct summand of $\widehat \Lambda$. 

The summand corresponding on the Fourier side to $\mathbf{L}$ is the Loomis-Whitney trilinear Brascamp-Lieb form, since we have 
\begin{equation*}
  \Lambda_\mathbf{L}(f_1, f_2, f_3) = \int f_{1}(x,u)f_{2}(y,v)f_{3}(x+v,y+u)K(u+v)\,dx\,dy\,du\,dv
\end{equation*}
\begin{equation*}%\label{3Loom1}
    = \int \widehat{f_{1}}(\xi_{1},\xi_{2}+\xi_{3})\widehat{f_{2}}(\xi_{2},\xi_{1}+\xi_{3})\widehat{f_{3}}(-\xi_{1},-\xi_{2})\widehat{K}(-\xi_{3})\,d\xi_{1}\,d\xi_{2}\,d\xi_{3}\,.
\end{equation*}
Estimating $\|\widehat{K}\|_\infty\leq 1$, changing variables $\xi_1+\xi_3+\xi_2=\tau$ and
shearing the functions $\widehat{f_1},\widehat{f_2}$,
this becomes exactly the Loomis-Whitney inequality trilinear form, which is then estimated by  
\[\|\widehat f_1\|_2\|\widehat f_2\|_2\|\widehat f_3\|_2 = \| f_1\|_2\| f_2\|_2\| f_3\|_2\,. \]

For the summands $\mathbf{B}$ we obtain similarly
\begin{equation*}
    \Lambda_\mathbf{B}(f_1, f_2, f_3) = \int f_{1}(x,y)f_{2}(x+z,y+u)f_{3}(x+v,z+u)K(u,v)\,dx\,dy\,dz\,du\,dv
\end{equation*}
\begin{equation*}%\label{3Loom2}
    = \int \widehat{f_{1}}(-\xi_{2}-\xi_{3},\xi_{1})\widehat{f_{2}}(\xi_{2},-\xi_{1})\widehat{f_{3}}(\xi_{3},-\xi_{2})\widehat{K}(\xi_{1}+\xi_{2},-\xi_{3})\,d\xi_{1}\,d\xi_{2}\,d\xi_{3}\,,
\end{equation*}
which after estimating  $\|\widehat{K}\|_\infty\leq 1$ and changing variables is again bounded by the Loomis-Whitney trilinear form of $\widehat f_1, \widehat f_2$ and $\widehat f_3$. For the summands $\mathbf{M}_i$ from Table \ref{table Holder} boundedness of the summands in $\widehat{\Lambda}$ reduces similarly to the Cauchy-Schwarz inequality.

\section{Proof of the projection theorems, Theorem \ref{thm summand} and Theorem \ref{thm n}}

To prove Theorem \ref{thm summand} and Theorem \ref{thm n}, we will need to extend Calderón-Zygmund kernels $K$ on some Hilbert space $H_0$ to kernels on a larger Hilbert space $H_0 \oplus H_0'$. The following lemma allows us to do that.

\begin{lemma}
    \label{lem check cz}
    Let $d, d' \ge 1$ and let $K$ be a Calderón-Zygmund kernel on $\R^d$. Define
    \[
        K'(x,y) =  |x|^{-d'} \exp\left(-\pi \frac{|y|^2}{|x|^2}\right) K(x)\,.
    \]
    For sufficiently small $c = c(d,d',l) > 0$,
    the kernel $c K'(x,y)$ is an $l$-Calderón-Zygmund kernel on $\R^{d+d'}$.
\end{lemma}

\begin{proof}
    Using that the assumptions on $K$ are invariant under dilations,
    it suffices by scaling to show that for $|\xi| = 1$, $|\eta| \le 1$ and for $|\eta| = 1$, $|\xi| \le 1$ and all $|\alpha| \le l$, we have 
    \begin{equation*}
%        \label{eq CZgoal}
        |\partial^\alpha \widehat{K}'(\xi, \eta)| \le 1/c\,.
    \end{equation*}
    Denote the heat kernel by $\Phi(\xi,t) = t^{-d/2}\operatorname{exp}(-\pi|\xi|^{2}/t)$. By a direct computation, we find that 
    \begin{equation}
        \label{eq K' hat}
        \widehat K'(\xi, \eta) = \int_{\R^d} \widehat K(u) \frac{1}{|\eta|^d} \exp\left(-\pi\frac{|\xi - u|^2}{|\eta|^2}\right) \, du = \int_{\R^d} \widehat K(u) \Phi(\xi - u,|\eta|^2) \, du\,.
    \end{equation}
    
    First, suppose that $|\eta| = 1, |\xi| \le 1$. The derivatives of the heat kernel take the form 
    \begin{equation}\label{partailPhi}
\partial^{\alpha}\Phi(\xi-u,|\eta|^{2})=|\eta|^{-|\alpha|}p\left(\frac{\xi-u}{|\eta|}, \frac{\eta}{|\eta|} \right)\Phi(\xi-u,|\eta|^{2})\,,
    \end{equation}
    where $p$ is a polynomial of degree $2|\alpha|$.
    Therefore, using \eqref{czkernel}, it holds for $|\alpha| \le l$
    \begin{equation*}
        |\partial^\alpha \widehat{K}'(\xi, \eta)|\le C |\eta|^{-|\alpha|}\|\widehat{K}\|_{\infty}\int_{\R^d} (1 + \lvert u \rvert^{2m}) e^{-\pi |u|^{2}}du\le C\,.
    \end{equation*}
    
    Now suppose that $|\eta| \leq 1$ and $|\xi| = 1$. We split up the integral in \eqref{eq K' hat}. Pick a smooth function $\varphi$ on $H_0$ with $\mathbf{1}_{B(0,1/4)} \le \varphi \le \mathbf{1}_{B(0,1/2)}$. Then
    \begin{align*}
        \widehat K'(\xi, \eta) &= \int  (1 - \varphi(u))\widehat K(u) \Phi(\xi - u,|\eta|^2) \, du  + \int \varphi(u)\widehat K(u)\Phi(\xi - u,|\eta|^2) \, du \\
        &= G_1(\xi, |\eta|^2) + G_2(\xi, |\eta|^2)\,.
    \end{align*} 
    Note that the function $G_1(\xi, t)$ solves the heat equation
    \[
        4 \pi \partial_t G_1(\xi, t) = \Delta_\xi G_1(\xi, t)\,.
    \]
    Using this to replace all derivatives in the second argument of $G_1$ by derivatives in $\xi$, we obtain
    \[
        \partial_\xi^\beta \partial_\eta^\gamma (G_1(\xi, |\eta|^2)) = \sum_{j = 1}^{|\gamma|} p_j(\eta) \partial_\xi^\beta \Delta^j_\xi G_1(\xi, |\eta|^2)\,,
    \]
    for certain polynomials $p_j$ that depend only on $\gamma$. It follows that
    \begin{align*}
        |\partial_\xi^\beta \partial_\eta^\gamma G_1(\xi, |\eta|^2)| &= \left|\sum_{j = 1}^{|\gamma|} \int p_j(\eta) \left(\partial_\xi^\beta \Delta^j_\xi ((1 - \varphi(\xi - u))\widehat K(\xi - u)) \right)\Phi(u,|\eta|^2) \, du\right|\\
        &\le \sum_{j=1}^{|\gamma|} \sup_{|\eta| \le 1} |p_j(\eta)| \sup_{u} \left| \partial_u^\beta \Delta^j_u ((1 - \varphi( u))\widehat K(u))\right|\,.
    \end{align*}
    Since $1 - \varphi$ is smooth and supported on the complement of $B(0,1/4)$ and since $|\beta| + 2|\gamma| \le 2l$, it follows from the Mikhlin condition \eqref{czkernel} for the $2l$-Calderón-Zygmund kernel $K$ that this is bounded by a constant depending only on $l$.
    
    The derivatives of the second term $G_2$ are given by
    \[
        \partial^\alpha G_2(\xi, |\eta|^2) = \int \varphi(u)\widehat K(u) \partial^\alpha \Phi(\xi - u, |\eta|^2)  \, du\,.
    \]
    On the support of the integrand, we have $|u| \le 1/2$ and hence $1/2\le |\xi - u| \le 3/2$. Further, we have $|\eta| \leq 1$. Using \eqref{partailPhi}, we obtain
    \begin{align*}
        \left| \partial^{\alpha}\Phi(\xi-u,|\eta|^{2})\right|&\le C |\eta|^{-|\alpha|}p\left(\frac{\xi-u}{|\eta|}, \frac{\eta}{|\eta|}\right)\Phi(\xi-u,|\eta|^{2})\\
        &\le C |\eta|^{-d - |\alpha|} \left( \frac{1}{|\eta|}+1\right)^{2|\alpha|}\exp\left(-\frac{\pi}{|\eta|^{2}}\right) \le C\,.
    \end{align*}
    Hence, we have
    \begin{equation*}
        |\partial^\alpha G_2(\xi, \eta)|\le C\|\widehat{K}\|_{\infty}\|\varphi\|_{1} \le C\,,
    \end{equation*}
    which completes the proof of the lemma.
\end{proof}

Next, we note that in proving Theorem \ref{thm summand} and Theorem \ref{thm n}, we may restrict attention to bounded Calderón-Zygmund kernels. Indeed, every Calderón-Zygmund kernel $K$ is the weak limit as $R\to \infty$ of the bounded kernels $K_R$ defined by
\[
    \widehat{K_R}(\xi) =  \widehat K(\xi) (\varphi(R^{-1}\xi)- \varphi(R\xi))\,,
\]
for smooth $\varphi$ with $\mathbf{1}_{B(0,1)} \le \varphi \le \mathbf{1}_{B(0,2)}$. Letting $R \to \infty$ in the conclusion of either theorem with kernel $K_R$ yields the conclusion for the general kernel $K$.

\begin{proof}[Proof of Theorem \ref{thm summand}]
We fix a bounded $2l$-Calderón-Zygmund kernel $K$ on $H_0$.  By assumption, for all $l$-Calderón-Zygmund kernels $K'$ on $H_0 \oplus H_0'$ and all functions $F_1, F_2, F_3$
\begin{equation}
    \label{eq sum bound}
    |\Lambda_{\mathbf{H} \oplus \mathbf{H}'}(K', F_1, F_2, F_3)| \leq C \|F_1\|_{p_1}\|F_2\|_{p_2} \|F_3\|_{p_3}\,.
\end{equation}
Our goal is to show that there exists $C'$ such that for all $f_1, f_2, f_3$
\[
    |\Lambda_{\mathbf{H}}(K, f_1, f_2, f_3)| \leq C'\|f_1\|_{p_1}\|f_2\|_{p_2} \|f_3\|_{p_3}\,.
\]
We choose orthonormal bases of the spaces $H_0, H_0'$, so that the inner product becomes the standard inner product on $\R^n$, and we denote the corresponding norm by $|\cdot|$.

We will apply \eqref{eq sum bound} to the kernel $K'$ on $H_0 \oplus H'_0$ obtained by extending $K$ as in Lemma \ref{lem check cz}, thus 
\begin{equation*}
%    \label{eq K' def}
    K'(x,y) =  |x|^{-\dim H_0'} \exp\left(-\pi \frac{|y|^2}{|x|^2}\right) K(x)\, .
\end{equation*}
We pick for $i = 1,2,3$,  functions $F_i^N$ on $H_i \oplus H_i'$ defined by
\[
     F_i^N(x,y) = f_i(x) N^{-\frac{\dim H_i'}{p_i}} \exp\left(-\pi\frac{|y|^2}{N^2}\right)\,.
\]
Since the datum $\mathbf{H} \oplus \mathbf{H}'$ is by assumption $\mathbf{p}$-bounded, we can apply \eqref{eq SBL nec 1} and \eqref{eq SBL nec 2} to the subspaces $\ker \Pi_0 \subseteq \ker (\Pi_0 \oplus\Pi_0')$, $\ker \Pi_0' \subseteq \ker (\Pi_0 \oplus \Pi_0')$,  and $\ker (\Pi_0 \oplus \Pi_0')$,  to obtain respectively
\begin{equation}
    \label{eq exponents}
    \sum_{i=1}^3 \frac{\dim H_i'}{p_i} \ge \dim H' - \dim H_0'\,,
\end{equation}
\[
    \sum_{i=1}^3 \frac{\dim H_i}{p_i} \ge \dim H - \dim H_0
\]
and 
\[
    \sum_{i=1}^3 \frac{\dim H_i' + \dim H_i}{p_i} = \dim H' + \dim H - \dim H_0' - \dim H_0\,.
\]
So we must have equality in \eqref{eq exponents}. 
Evaluating $\Lambda_{\mathbf{H} \oplus \mathbf{H'}}$ with our choice of functions and kernel yields then
\[
    \Lambda_{\mathbf{H}\oplus \mathbf{H}'}(K', F_1^N,  F_2^N,  F_3^N) = \int  \prod_{i=1}^3 f_i(\Pi_i(x))  K(\Pi_0(x))  N^{-\dim H' + \dim H_0'} |\Pi_0(x)|^{-\dim H_0'}
\]
\begin{equation}
    \times\int_{H'} \exp\left(-\pi \frac{|\Pi_1'(y)|^2 + |\Pi_2'(y)|^2  + |\Pi_3'(y)|^2}{N^2} - \pi \frac{|\Pi_0'(y)|^2}{|\Pi_0(x)|^{2}}\right) \, dy \, dx\,. \label{eq gaussian}
\end{equation}
Denote by $A(x)$ the matrix
\[
    A(x) =  \frac{1}{|\Pi_0(x)|^2}{\Pi_0'}^t \Pi_0' + \frac{1}{N^2}({\Pi_1'}^t \Pi_1' + {\Pi_2'}^t \Pi_2' + {\Pi_3'}^t \Pi_3')  = \frac{1}{|\Pi_0(x)|^2} A_0 + \frac{1}{N^2} A_1\,.
\]
Then the Gaussian $y$-integral in \eqref{eq gaussian} evaluates to $\det(A(x))^{-1/2}$, so
\begin{equation}
    \label{eq integrand dom}
    \eqref{eq gaussian} = \int  \prod_{i=1}^3 f_i(\Pi_i(x)) K(\Pi_0(x)) N^{-\dim H' + \dim H_0'} |\Pi_0(x)|^{-\dim H_0'} \det(A(x))^{-1/2} \, dx\,.
\end{equation}
We claim that there exists constants $ c(\mathbf{H}'), C(\mathbf{H}')$ with
\begin{equation}
    \label{eq determinant}
    \det A(x) = c(\mathbf{H}') \cdot |\Pi_0(x)|^{-2 \dim H_0'} N^{-2 (\dim H' - \dim H_0')} \cdot (1 + O(N^{-2}|\Pi_0(x)|^2))
\end{equation}
and 
\begin{equation}
    \label{eq determinant upper}
    \det(A(x))^{-1/2} \le C(\mathbf{H}') \cdot |\Pi_0(x)|^{\dim H_0'} N^{(\dim H' - \dim H_0')}\,.
\end{equation}
Suppose \eqref{eq determinant} and \eqref{eq determinant upper} hold. Because $f_1, f_2$ and $f_3$ are Schwartz functions and $K$ is bounded, it follows that the integrand in \eqref{eq integrand dom} is uniformly in $N$ controlled by an integrable function of $x$.
Using \eqref{eq determinant} in \eqref{eq integrand dom} and sending $N \to \infty$, we obtain with the dominated convergence theorem 
\[
    \Lambda_{\mathbf{H}}(K, f_1, f_2, f_3) = c(\mathbf{H}') \lim_{N \to \infty} \Lambda_{\mathbf{H}\oplus \mathbf{H}'}(K', F_1^N,  F_2^N,  F_3^N)\,.
\]
With the boundedness assumption \eqref{eq sum bound} on $\Lambda_{\mathbf{H} \oplus \mathbf{H}'}$, it follows that there exist constants $C', C'' > 0$ such that
\[
   | \Lambda_{\mathbf{H}}(K, f_1, f_2, f_3) |
    \le C' \limsup_{N \to \infty}\|F_1^N\|_{p_1} \|F_2^N\|_{p_2} \|F_3^N\|_{p_3}= C''\|f_1\|_{p_1}\|f_2\|_{p_2} \|f_3\|_{p_3}\,.
\]
This completes the proof, up to verifying \eqref{eq determinant} and \eqref{eq determinant upper}.

To show \eqref{eq determinant}, we may assume by a base change that $A_0$ is a diagonal matrix, so that 
\[
    A(x) = \frac{1}{|\Pi_0(x)|^2} \begin{pmatrix}
        \lambda_1 & 0 & \dots & 0 & 0&\dots & 0\\
        0 & \lambda_2 & \dotsc & 0 & 0&\dots & 0\\
        \vdots & \vdots &\ddots & \vdots & \vdots & \dots & \vdots \\
        0 & 0 & \dots & \lambda_{\dim H_0'} &0  & \dots & 0\\
        0 & 0 & \dots &0 &0 &\dots & 0\\
        \vdots & \vdots & & \vdots & \vdots & \ddots &\vdots \\
        0 & 0 & \dots &0 & 0&\dots & 0
    \end{pmatrix} + \frac{1}{N^2}A_1\,,
\]
where $\lambda_1, \dotsc, \lambda_{\dim H_0'}$ are the nonzero eigenvalues of ${\Pi_0'}^t\Pi_0'$.
Then \eqref{eq determinant} follows by expanding $\det A(x)$ using the Leibniz formula: The diagonal in $A_0$ contributes the first term, while the contribution of all other terms is controlled by the $O(N^{-2} |\Pi_0(x)|^2)$ term. 

To show \eqref{eq determinant upper}, we can assume $N \le C|\Pi_0(x)|$ for a sufficiently large constant $C$, since otherwise it already follows from \eqref{eq determinant}. But then we have, since $A_0 \ge 0$ 
\begin{equation*}
    \det A(x) \ge \det(\frac{1}{N^2}A_1) = N^{-2 \dim H'} \det(A_1)
\end{equation*}
\begin{equation}
    \label{eq determinant upper 1}
    \ge C^{-2 \dim H_0'} |\Pi_0(x)|^{-2\dim H_0'} N^{-2(\dim H' - \dim H_0')} \det(A_1)\,.
\end{equation}
Note that $\det(A_1) > 0$ because $\ker \Pi_1' \cap  \ker \Pi_2' \cap \ker \Pi_3' = \{0\}$, which follows from \eqref{eq SBL nec 2}. Taking \eqref{eq determinant upper 1} to the power $-1/2$ then gives \eqref{eq determinant upper}.
\end{proof}

\begin{proof}[Proof of Theorem \ref{thm n}]

We proceed similarly as in the proof of Theorem \ref{thm summand}. We fix again a bounded $2l$-Calderón-Zygmund kernel $K$. 
We further fix $n$ and assume that there exists $C > 0$ such that for each $l$-Calderón-Zygmund kernel $K'$ 
\begin{equation*}
%    \label{eq assum n}
    |\Lambda_{\mathbf{H}_n}(K', F_1, F_2, F_3)| \leq C \|F_1\|_{p_1}\|F_2\|_{p_2} \|F_3\|_{p_3}\,.
\end{equation*}
Our goal is to show that there exists $C'$ such that for all $f_1, f_2, f_3$ 
\[
    |\Lambda_{\mathbf{H}_{n-1}}(K, f_1, f_2, f_3)| \leq C' \|f_1\|_{p_1}\|f_2\|_{p_2} \|f_3\|_{p_3}\,.
\]
Suppose that the singular Brascamp-Lieb datum $\mathbf{H}_{n-1}$ associated with $\mathbf{M}\oplus \mathbf{J}^{(2)}_{n-1}$ is
\[
    (H \oplus \R^{n-1} \oplus \R^{n-1}, H_0 \oplus \R^{n-1}, H_1 \oplus \R^{n-1}, H_2 \oplus \R^{n-1}\,, H_3 \oplus \R^{n-1},\Pi_0, \Pi_1, \Pi_2, \Pi_3)\,,
\]
for linear maps $\Pi_i$. Comparing the matrices associated with $\mathbf{J}^{(2)}_n$ and $\mathbf{J}^{(2)}_{n-1}$ in Table \ref{table htype} shows that the datum associated with $\mathbf{H}_n$ is then given by
\[
    (H \oplus \R^{n} \oplus \R^n, H_0 \oplus \R^{n}, H_1 \oplus \R^{n}, H_2 \oplus \R^{n}\,, H_3 \oplus \R^{n}, \Pi_0', \Pi_1', \Pi_2', \Pi_3')\,,
\]
where we have, writing $x\in H,\, (y, y_n)\in \R^{n-1}\times \R,\, (z, z_n) \in   \R^{n-1}\times \R$:
\[
    \Pi_0'(x,y, y_n , z, z_n) = (\Pi_0(x,  y,  z), y_n)\,,
\]
\[
    \Pi_1'(x,y, y_n , z, z_n) = (\Pi_1(x,  y,  z), z_n)\,,
\]
\[
    \Pi_2'(x, y, y_n ,  z, z_n) = (\Pi_2(x,  y, z), z_n+y_n)\,,
\]
\[
    \Pi_3'(x, y, y_n ,  z, z_n) = (\Pi_3(x,  y,  z), z_n + y_{n-1})\,.
\]
We define for $i = 1,2,3$
\[
    F_i(x, z, z_n) = N^{-1/p_i} \exp\left(-\pi \frac{ z_n^2}{N^2}\right) f_i(x, z)\,,
\]
and we set
\[
    K'(x,y, y_n) = |(x,y)|^{-1} \exp\left(- \pi \frac{y_n^2}{|(x, y)|^2}\right) K(x,y)\,.
\]
By Lemma \ref{lem check cz}, the kernel $K'$ is an $l$-Calderón-Zygmund kernel. We have that
\[
    \Lambda_{\mathbf{H}_n}(K', F_1, F_2, F_3) = \int_{H \oplus \R^{n-1} \oplus \R^{n-1}}  \prod_{i=1}^3 f_i(\Pi_i(x,y,z)) K(x,y)   
\]
\[
    \times  \int_\R \int_\R \frac{1}{|(x,y)| N} \exp\left(-\pi \frac{z_n^2 + (z_n + y_n)^2 + (z_n + y_{n-1})^2}{N^2} - \pi \frac{y_n^2}{|(x,y)|^2}\right)  \, dy_n \, dz_n \, dx \, dy \, dz\,.
\]
The $z_n$ integral can be evaluated by first expanding $(z_n+y_n)^2,\,(z_n+y_{n-1})^2$ and then completing the square in $z_n$. One obtains that the inner two integrals equal
\[
     \frac{1}{|(x,y)|}\frac{1}{\sqrt{3}} \int_\R   \exp\left(-\pi \frac{2(y_{n-1}^2 - y_{n-1}y_n + y_n^2)}{3N^2} - \pi \frac{y_n^2}{|(x,y)|^2}\right)  \, dy_n\,.
\]
This integral is bounded by, and converges by monotone convergence as $N \to \infty$, to 
\[
    \frac{1}{|(x,y)|}\frac{1}{\sqrt{3}} \int_\R  \exp\left(-\pi \frac{y_n^2}{|(x,y)|^2}\right) \, dy_n = \frac{1}{\sqrt{3}}\,.
\]
Using that $K$ is bounded and that $f_1, f_2, f_3$ are Schwartz functions, we obtain with the dominated convergence theorem
\[
    \Lambda_{\mathbf{H}_{n-1}}(K, f_1, f_2, f_3) = \lim_{N \to \infty} \sqrt{3} \Lambda_{\mathbf{H}_n}(K', F_1, F_2, F_3) \,.
\]
Combined with boundedness of $\Lambda_{\mathbf{H}_n}$ this shows that there exist constants $C', C'' > 0$ with
\[
    |\Lambda_{\mathbf{H}_{n-1}}(K, f_1, f_2, f_3)| \le C' \limsup_{N \to \infty} \|F_1\|_{p_1} \|F_2\|_{p_2} \|F_3\|_{p_3}= C'' \|f_1\|_{p_1} \|f_2\|_{p_2} \|f_3\|_{p_3}\,.\qedhere
\]
\end{proof}

\section{Method of rotations: Proof of Theorem \ref{thm rot}}
\label{sec rot}
Fix the dimension $d \ge 3$. We denote $S^{d-1} = \{\theta \in \R^d \, : \, |\theta| = 1\}$, and we denote by $\sigma$ the normalized $(d-1)$-dimensional Hausdorff probability measure on $S^{d-1}$.
Further, if $\nu \in S^{d-1}$, then we denote by 
$\sigma_{\nu}$ the normalized $(d-2)$-dimensional Hausdorff probability measure on the great circle
\[
    (\lin \nu)^{\bot} \cap S^{d-1}\,.
\]

Recall that there is an orthogonal decomposition
\[
    L^2(S^{d-1}) = \overline{\bigoplus_{n = 0}^\infty \mathcal{H}_n} \, ,
\]
where $\mathcal{H}_n$ is the space of spherical harmonics of degree $n$ on $S^{d-1}$, see e.g. \cite[Chapter IV]{SteinWeiss}. Another way to characterize $\mathcal{H}_n$ is as the space of eigenfunctions of the spherical Laplacian corresponding to the eigenvalue $\lambda_n = -n(n + d - 2)$. 

We will use the spherical Sobolev spaces $H^s(S^{d-1})$ defined by 
\[
\mathcal{H}^s(S^{d-1}) =\{f \in L^2(S^{d-1}) \, : \, \|f\|^2_{H^s(S^{d-1})} = \sum_{n = 0}^\infty \lambda_n^s \|\pi_{n}(f)\|_{L^2(S^{d-1})}^2 < \infty \}\,,
\]
where $\pi_{n}$ denotes the orthogonal projection onto $\mathcal{H}_n$. We will also use the Funk transform, which is the operator $T$ defined a priori on continuous functions $F$ on $S^{d-1}$ by
\begin{equation}
    \label{eq funk def}
     TF(\theta) = \int F(\nu) \, d\sigma_\theta(\nu)\,.
\end{equation}
We will need the following properties of the Funk transform. 
\begin{lemma}
    \label{lem Funk properties}
    Let $H_0^s(S^{d-1})$ be the space of functions in the smoothness $s$ Sobolev space on $S^{d-1}$ of mean zero.
    For all $s \ge 0$, the Funk-transform $T$ extends to a contraction 
    \[
        T\,:\, H_0^s(S^{d-1}) \to H_0^s(S^{d-1})\,, \qquad \|T\|_{H^s_0 \to H^s_0} = \frac{1}{d-1} < 1\,.
    \]
    Moreover, for all $s \ge 0$, the operator $T$ extends to a bounded operator 
    \[
        T\,:\, H^s_0(S^{d-1}) \to H_0^{s + \delta}(S^{d-1})\,,
    \]
    where  $\delta = \frac{d-2}{2}$.
\end{lemma}

The proof of Lemma \ref{lem Funk properties} relies on the Funk-Hecke formula.

\begin{lemma}[Funk-Hecke formula]
    \label{lem FH}
    Denote by $\omega_m$ the $m$-dimensional Hausdorff measure on $S^{m}$. Let $f: [-1,1] \to \mathbb{R}$ be a continuous function. Then for every spherical harmonic $Y_n$ of degree $n$ and $\theta\in S^{d-1}$, 
    \[
        \int_{S^{d-1}} f(\nu \cdot \theta) Y_n(\nu) \, d\sigma(\nu) = \frac{\omega_{d-2}}{\omega_{d-1}} \lambda_n Y_n(\theta)\,,
    \]
    where 
    \[
        \lambda_n =   \int_{-1}^1 \frac{C_n^{\frac{d-2}{2}}(t)}{C_n^{\frac{d-2}{2}}(1)} f(t)  (1 - t^2)^{\frac{d-3}{2}} \, dt\,.
    \]
    Here $C_n^k(t)$ denotes the Gegenbauer polynomials, defined via the generating function
    \begin{equation}
        \label{eq gen func}
        (1 - 2rt + r^2)^{-k} = \sum_{n \ge 0} C_n^k(t) r^n\,.
    \end{equation}
\end{lemma}

\begin{proof}
    See for example \cite{FengYuan}, Theorem 1.2.9.
\end{proof}

\begin{proof}[Proof of Lemma \ref{lem Funk properties}]
    Let $(f_k)$ be a sequence of continuous functions such that $f_k$ is supported in $(-1/k, 1/k)$ and $\int_{-1/k}^{1/k} f_k(t) \, dt = 1$.
    A computation in coordinates shows that for every $\theta \in S^{d-1}$ 
    \[
        \sigma_\theta(\nu) = \frac{\omega_{d-1}}{\omega_{d-2}} \lim_{k \to \infty}  f_k(\nu \cdot \theta)\sigma(\nu)\,,
    \]
    in the sense of weak convergence of measures. Applying Lemma \ref{lem FH} to the sequence $f_k$ and taking limits, we obtain that 
    \[
        T Y_n = \lambda_n Y_n
    \]
    for every spherical harmonic $Y_n$ of degree $n$, where 
    \[
        \lambda_n = \frac{C_n^{\frac{d-2}{2}}(0)}{C_n^{\frac{d-2}{2}}(1)}\,.
    \]
    We compute the values of $C_n^{\frac{d-2}{2}}$ in $0$ and $1$ using \eqref{eq gen func}. Note the identity
    \begin{equation}
        \label{eq geometric}
        \frac{1}{(1 - x)^k} = \sum_{n = 0}^\infty \binom{n + k - 1}{k-1} x^n\,. 
    \end{equation}
    Combining \eqref{eq gen func} and \eqref{eq geometric}, we have 
    \[
        C_n^{\frac{d-2}{2}}(0) = 
        \begin{cases}
            (-1)^{n/2}\binom{\frac{n}{2} + \frac{d-2}{2} - 1}{\frac{d-2}{2} - 1} \quad &\text{if $n$ is even}\\
            0 \quad &\text{if $n$ is odd}\,,
        \end{cases}
    \]
    and 
    \[
        C_n^{\frac{d-2}{2}}(1) = \binom{n + d - 3}{d-3}\,.
    \]
    Hence $|\lambda_n|$ clearly vanishes for odd $n$. For even $n$ we obtain with the duplication formula $\Gamma(z)\Gamma(z + \frac{1}{2}) = \sqrt{\pi} 2^{1 - 2z}\Gamma(2z)$ and Stirling's formula
    \begin{equation}
        \label{eq lambdan}
        |\lambda_n| = \frac{\Gamma(\frac{n}{2} + \frac{d}{2} - 1)\Gamma(n+1)\Gamma(d-2)}{\Gamma(\frac{n}{2}+1)\Gamma(n + d - 2)\Gamma(\frac{d}{2}-1)} = 2^{3 - d} \frac{\Gamma(d-2)\Gamma(\frac{n+1}{2})}{\Gamma(\frac{d}{2}-1)\Gamma(\frac{n+d-1}{2})} = O(n^{\frac{2-d}{2}})\,.
    \end{equation}
    Thus $T$ maps $H^s_0$ into $H^{s + \delta}_0$, for $\delta = \frac{d-2}{2} > 0$. Equation \eqref{eq lambdan} combined with logarithmic convexity of the $\Gamma$-function also shows that $|\lambda_{2n}|$ is decreasing, so that 
    \[
        \|T\|_{H_0^s \to H_0^s} = |\lambda_2| =\frac{\Gamma(\frac{d}{2})\Gamma(3)\Gamma(d-2)}{\Gamma(2)\Gamma(d)\Gamma(\frac{d}{2}-1)} = \frac{1}{d-1}\,. 
    \]
\end{proof}

We define the manifold of all pairs of orthogonal vectors in $S^{d-1}$
\[
    \mathcal{M}_d = \{(\nu, \theta) \in S^{d-1} \times S^{d-1} \, : \, \theta \cdot \nu = 0\}\,.
\]
Below we will make use of the fact that the normalized Hausdorff measure on $\mathcal{M}_d$ disintegrates as
\begin{equation}
\label{eq desint}
    d\sigma_{\theta}(\nu) d\sigma(\theta) =  d\sigma_{\nu}(\theta) d\sigma(\nu)\,.
\end{equation}
Theorem \ref{thm rot} is a consequence of the following key proposition. 

\begin{proposition}
    \label{prop sphere}
    Let $d \ge 3$ and $s > 1/2$. There exists a constant $C > 0$ such that the following holds. Let $\Omega \in H_0^s(S^{d-1})$.
    Then there exists a function $\Gamma: \mathcal{M}_d \to \mathbb{C}$ such that
    \begin{itemize}
        \item for all $\nu \in S^{d-1}$
        \begin{equation}
            \label{eq mean 0}
            \int_{(\lin \nu)^{\bot} \cap S^{d-1}} \Gamma(\nu, \theta) \, d\sigma_{\nu}(\theta) = 0\,.
        \end{equation}
        and 
        \begin{equation}
            \label{eq Sobolev estimate}
            \|\Gamma(\nu, \cdot)\|_{H_0^{s-1/2}((\lin \nu)^{\bot} \cap S^{d-1})} \le C \|\Omega\|_{H_0^s(S^{d-1})}\,.
        \end{equation}
        \item as measures, we have
        \begin{equation}
            \label{eq repr}
            \Omega(\theta)\sigma(\theta) = \int_{S^{d-1}} \Gamma(\nu, \theta)\sigma_{\nu}(\theta) \, d\sigma(\nu)\,.
        \end{equation}
    \end{itemize}
    Moreover, $\Gamma$ can be chosen so that the mapping $\Omega \mapsto \Gamma$ is continuous from $C^k(S^{d-1})$ into $C^k(\mathcal{M}_d)$, for every $k$.
\end{proposition}

Proposition \ref{prop sphere} says that any mean zero function on $S^{d-1}$ can be decomposed into mean zero functions on slices $S^{d-1} \cap (\lin \nu)^{\bot}$. We will use this later to decompose Calderón-Zygmund kernels on $\R^d$ into kernels on $(\lin \nu)^{\bot}$, $\nu \in S^{d-1}$.

\begin{proof}[Proof of Proposition \ref{prop sphere}]
    For $F: S^{d-1} \to \mathbb{C}$, we define a candidate solution to  \eqref{eq mean 0}, \eqref{eq repr} by 
    \[
        \Gamma[F] \, : \, \mathcal{M}_d \to \mathbb{C}, \qquad \Gamma[F](\nu, \theta) = F(\theta) - \int F(\gamma) \, d\sigma_{\nu}(\gamma)\,.
    \]
    The function $\Gamma[F]$ satisfies \eqref{eq mean 0} by construction. On the other hand, it satisfies \eqref{eq repr} if and only if, as measures,
    \begin{equation*}
        \Omega(\theta) \sigma(\theta) =  \int \left[ F(\theta) - \int F(\gamma) \, d\sigma_{\nu}  (\gamma) \right]\sigma_{\nu}(\theta) \, d\sigma(\nu)\,.
    \end{equation*}
    Using \eqref{eq desint} and the definition \eqref{eq funk def} of $T$ to simplify the second summand, we obtain
    \[
        = F(\theta) \sigma(\theta) - \int TF(\nu) \, d\sigma_\theta(\nu) \, \sigma(\theta) = (F(\theta) - T^2 F(\theta)) \sigma(\theta)\,.
    \]
    Thus, \eqref{eq repr} holds if 
    \begin{equation}
        \label{eq solution}
        \Omega = (1 - T^2)F\,.
    \end{equation}
    Lemma \ref{lem Funk properties} now implies that $1-T^2$ is invertible on $H_0^s(S^{d-1})$ for all $s > 1/2$ and hence \eqref{eq solution} can be solved for $F$ for every $\Omega \in H^s_0$, and the solution map is continuous. The function $\Gamma(\nu, \cdot)$ is up to a constant the restriction of $F$ to the codimension one submanifold $S^{d-1} \cap (\lin \nu)^{\bot}$. Since $F \in H_0^s$, we obtain with the trace theorem \eqref{eq Sobolev estimate}.    
    Finally, we have for $\Omega \in C^k$ that 
    \[
        F = \sum_{l = 0}^\infty T^{2l} \Omega = \sum_{l = 0}^2 T^{2l} \Omega + \sum_{l = 3}^\infty T^{2l}\Omega\,.
    \]
    The first three terms on the right hand side are in $C^k$, since $T$ maps $C^k$ into $C^k$. By Lemma \ref{lem Funk properties} we have $T^{6} \Omega\in H^{k + d}$. Since $\|T\|_{H_0^{k+d}\to H_0^{k+d}} < 1$, the second sum converges in  $H^{k + d}$, which embedds into $C^k$. Thus the solution map is continuous on $C^k$.
\end{proof}

We will apply Proposition \ref{prop sphere} to the restriction of a homogeneous Calderón-Zygmund kernel to the sphere $S^{d-1}$. Since our defining assumptions \eqref{czkernel} on Calderón-Zygmund kernels are formulated on the Fourier side, we need the following lemma to pass to kernels with prescribed smoothness in space.

\begin{lemma}[{\cite[Chapter IV, Theorem 4.7]{SteinWeiss}}]
    \label{lem FT sphere}
    Let $s \ge d$. Let $\Omega \in H_0^s(S^{d-1})$ be a mean zero function on $S^{d-1}$. Then $m(\xi) = \Omega(\xi/|\xi|)$ defines a homogenous of degree $0$ tempered distribution on $\R^d$. The inverse Fourier transform of $m$ is a homogenous of degree $-d$ tempered distribution on $\R^d$ which can be written as
    \begin{equation*}
%        \label{eq FT sphere}
        \widecheck m(x) = \Omega^*(x/|x|) |x|^{-d}\,.
    \end{equation*}
    The mapping $\Omega \mapsto \Omega^*$ is bounded with bounded inverse from $H_0^s(S^{d-1})$ into $H_0^{s-d}(S^{d-1})$.
\end{lemma}

\begin{proof}
    See Theorem 4.7 in Chapter IV of \cite{SteinWeiss}.
\end{proof}

\begin{proof}[Proof of Theorem \ref{thm rot}]
    Let $K$ be a homogenous $l$-Calderón-Zygmund kernel and let $\Omega: S^{d-1} \to \mathbb{C}$ be the function satisfying
    \[
        \widehat K(\xi) = \Omega(\xi/|\xi|) + C_0\,, \qquad \int \Omega(\theta) \, d\sigma(\theta) = 0\,.
    \]
    Since $K$ is an $l$-Calderón-Zygmund kernel, $\Omega \in C^l(S^{d-1})$, so in particular $\Omega \in H^l_0(S^{d-1})$. 
    By Lemma \ref{lem FT sphere}, the kernel $K$ is then given by
    \[
        K(x) = C_0 \delta + \Omega^*(x/|x|)|x|^{-d}\,,
    \]
    and $\Omega^*$ satisfies $\|\Omega^*\|_{H^{l-d}_0} \le C\|\Omega\|_{H^l_0}$. We apply Proposition \ref{prop sphere} to $\Omega^*$, using that $l \ge d+1$. We obtain for each $\nu \in S^{d-1}$ a function 
    \[
        \Omega_\nu^*(\theta) = \Gamma(\nu, \theta)
    \]
    such that 
    \begin{equation}
        \label{eq superposition 1}
        \Omega^*(\theta)\sigma(\theta) = \int_{S^{d-1}} \Omega_\nu^*(\theta) \sigma_\nu(\theta) \, d\sigma(\nu)\,,
    \end{equation}
    and such that 
    \[
        \|\Omega_\nu^*\|_{H^{l - d - 1/2}_0} \le C \|\Omega^*\|_{H^{l-d}_0} \le C \|\Omega\|_{H^l_0}\,.
    \]
    We define the kernel $K_\theta(x) = C_0\delta + \frac{\omega_{d-1}}{\omega_{d-2}} |x|^{1-d}\Omega_\theta^*(x/|x|)$ on $(\lin \theta)^{\bot}$. By applying Lemma \ref{lem FT sphere}, in the opposite direction, to the function $\Omega_\theta^*$, we find that for all $\xi \in (\lin \theta)^{\bot}$
    \[
        \widehat K_\theta(\xi) = C_0 + \frac{\omega_{d-1}}{\omega_{d-2}} \Omega_\theta(\xi/|\xi|),
    \]
    for a function $\Omega_\theta$ with
    \[
        \|\Omega_\theta\|_{H^{l - 3/2}_0(S^{d-1} \cap (\lin \theta)^{\bot}} \le C \|\Omega_0\|_{H^l_0(S^{d-1})}\,.
    \]
    Applying finally the Sobolev embedding theorem, we find that there exists a constant $C > 0$ such that for each $\theta \in S^{d-1}$, the kernel $C^{-1} K_\theta$ is an $l - \left\lceil\frac{d + 2}{2}\right\rceil$-Calderón-Zygmund kernel on $(\lin \theta)^{\bot}$.

    From \eqref{eq superposition 1}, we obtain for each Schwartz function $f$ on $\R^d$ by integration in polar coordinates:
    \begin{align}
        \int f(x) K(x) \, dx &= C_0 f(0) + \omega_{d-1} \int_0^\infty   \int_{S^{d-1}} f(r\nu) \Omega^*(\nu) \, d\sigma(\nu) \, \frac{dr}{r}\nonumber \\
        &= C_0 f(0) + \omega_{d-1} \int_0^\infty  \int_{S^{d-1}} \int_{S^{d-1}} f(r\nu) \Omega^*_\theta(\nu) \sigma_\theta(\nu) \, d\sigma(\theta) \, \frac{dr}{r}\nonumber \\
        &= \int_{S^{d-1}} C_0 f(0) + \omega_{d-1} \int_0^\infty  \int_{S^{d-1}} f(r\nu) \Omega^*_\theta(\nu) \sigma_\theta(\nu) \, \frac{dr}{r} \, d\sigma(\theta)\nonumber \\
        &= \int_{S^{d-1}} C_0 f(0) + \frac{\omega_{d-1}}{\omega_{d-2}} \int_{(\lin \theta)^{\bot}} f(x) \Omega_\theta^*(x/|x|) |x|^{1-d} \, d\mu_\theta(x) \, d\sigma(\theta)\,\nonumber\\
        &= \int_{S^{d-1}} \int_{(\lin \theta)^{\bot}} f(x) K_\theta(x) \, d\mu_\theta(x) \, d\sigma(\theta)\, .\label{eq superposition}
    \end{align}
    Here $\mu_\theta$ is the $d-1$-dimensional Lebesgue measure on $(\lin \theta)^{\bot}$. Combining \eqref{eq superposition} with Fubini's theorem, it follows that for all Schwartz functions $f_1, f_2, f_3$
    \[
        \Lambda_\mathbf{H}(K, f_1, f_2, f_3) = \int_{S^{d-1}} \Lambda_{\mathbf{H}(\theta)}(K_\theta, f_1, f_2, f_3) \, d\sigma(\theta)\,.
    \]
    Together with the triangle inequality and the assumption of integrable boundedness of the forms on the right hand side, this completes the proof.
\end{proof}

\section{Proof of Theorem \ref{thm type 03}}

We will deduce Theorem \ref{thm type 03} from the following multilinear multiplier bound from \cite{fraccaroli+2024}.  
Note that the condition \eqref{eq singular away subspace} on the multiplier is slightly more general than that obtained by translating the condition on the kernel of the corresponding singular Brascamp-Lieb forms. We will exploit this to deduce also bounds for forms with less singular multiplier. 

\begin{theorem}[{\cite[Theorem 1.1]{fraccaroli+2024}}]
    \label{thm nondegenerate}
     Let $2 < p_1, p_2, p_3 < \infty$ with $\frac{1}{p_1} + \frac{1}{p_2} + \frac{1}{p_3} = 1$ and fix a dimension $d\geq 1$. There exists $l\in \mathbb{N}$ such that the following holds. Let $\Gamma$ be the $2d$-dimensional subspace $\{(\xi_1, \xi_2, \xi_3) \in \R^{3d} \ : \ \xi_1 + \xi_2 + \xi_3 = 0\}$ of $\R^{3d}$. Furthermore, let $\Gamma' \subseteq \Gamma$ be a $d$-dimensional subspace, which can be parametrized in terms of each $\xi_1, \xi_2$ and $\xi_3$. There exists a constant $C$ such that the following holds. Let $M: \Gamma \to \mathbb{C}$ satisfy
    \begin{equation}
        \label{eq singular away subspace}
       |\partial^\alpha M(\xi)| \le (\dist(\xi, \Gamma'))^{-|\alpha|}\,, \qquad |\alpha| \le l\,.
    \end{equation}
    Then 
     \[
        \int_{\Gamma} \widehat f_1(\xi_1) \widehat f_2(\xi_2) \widehat f_3(\xi_3) M(\xi) \, d\mu_\Gamma(\xi) \le C\|f_1\|_{p_1} \|f_2\|_{p_2} \|f_3\|_{p_3}\,.
     \]
\end{theorem}

Theorem \ref{thm nondegenerate} is a special case of the main result in \cite{fraccaroli+2024}.
It can also be deduced by following the proof of Theorem 2.1.1 in \cite{RoosThesis}:
    By a reduction similar to, but slightly more general than in Section 2.2 of \cite{RoosThesis}, Theorem \ref{thm nondegenerate} reduces to a slightly more general version of Lemma 2.3.1 in \cite{RoosThesis}. The only difference to \cite{RoosThesis} is that the bump functions $\varphi_{y,\eta,t}$ are possibly different bump functions adapted to position $y$, frequency $\eta$ and scale $t$, as opposed to dilation of a fixed bump function. This causes no issues, because the assumptions on $M$ still guarantee the uniform estimates on $\varphi_{y,\eta,t}$ required in the proof.

\begin{proof}[Proof of Theorem \ref{thm type 03}]
 Choosing coordinates and swapping the role of $f_1$ and $f_3$, we may  express the trilinear form $\Lambda_\mathbf{H}$ up to a constant as 
    \begin{equation}
        \label{eq type tf}
         \int  f_1(x_1 + y_1, x_2 + C_1^Ty_2) f_2(x_1 + B^Ty_1, x_2 + C_2^Ty_2) f_3(x_1,x_2)K(y_1, y_2) \, dy_1 \, dy_2 \, dx_1 \, dx_2\,.
    \end{equation}
    Here $B$ is the direct sum of the matrices $X$ in the modules $\mathbf{N}_{n_i}$ occuring as direct summands in $\mathbf{M}_\mathbf{H}$. In particular, and that is all we will need, $B$ and $I - B$ are invertible. The matrices $C_1$ and $C_2$ are direct sums of $I_{m_i}^\uparrow$ and $I_{m_i}^\downarrow$ respectively, from the direct summands $\mathbf{C}_{m_i}$. 
    
    By Fourier inversion, \eqref{eq type tf} equals a multiple of
    \[
        \int_\Gamma \widehat f_1(\xi_{1,1}, \xi_{1,2}) \widehat f_2(\xi_{2,1}, \xi_{2,2})  \widehat f_3(\xi_{3,1}, \xi_{3,2})  \widehat K(-\xi_{1,1} - B \xi_{2,1}, -C_1 \xi_{1,2} - C_2 \xi_{2,2}) \, d\mu_\Gamma(\xi)\,.
    \]
    Define the singular subspace
    \[
        \Gamma_s = \{\xi \in \Gamma \ : \ \xi_{1,1} +B\xi_{2,1} =0, \, C_1 \xi_{1,2} +C_2 \xi_{2,2}=0\}\,.
    \]
    Define $D_1$ and $D_2$ to be the direct sum of $Q_{m_i}I_{m_i}^\uparrow$ and $Q_{m_i}I_{m_i}^\downarrow$, respectively, where $Q_m$ is the $m \times (m+1)$ matrix
    \[
        Q_m = 
        \begin{pmatrix}
            0 & 1 & 0 & \dots & 0\\
            0 & 0 & 1 & &  \\
            \vdots & \vdots &  & \ddots \\
            -1 & 0 & 0 & \dots & 1
        \end{pmatrix}\,.
    \]
    Then, the space $\Gamma_s$ sits in the larger subspace 
    \[
        \Gamma' = \{\xi \in \Gamma \ : \ \xi_{1,1} +B\xi_{2,1}= 0, \, D_1 \xi_{1,2}+ D_2 \xi_{2,2} =0\}\,. 
    \]
    The space $\Gamma'$ satisfies the nondegeneracy condition of Theorem \ref{thm nondegenerate}, because $B$ and $I- B$ are invertible and because $D_1, D_2$ and $D_1 - D_2$ are invertible. Indeed, this can be checked blockwise. The blocks in $D_1$ are of the form $Q_{m_i}I_{m_i}^\uparrow$, the blocks in $D_2$ are of the form $Q_{m_i}I_{m_i}^\downarrow$, and the blocks in  $D_1 - D_2$ are of the form $Q_{m_i}I_{m_i}^\uparrow - Q_{m_i}I_{m_i}^\downarrow$, each of which are invertible. We choose
    \[
        M(\xi) = \widehat K(-\xi_{1,1} - B \xi_{2,1}, -C_1 \xi_{1,2} - C_2 \xi_{2,2})\,.
    \]
    Then we have for $|\alpha| \le m$, by \eqref{czkernel}
    \begin{align*}
        |\partial^\alpha M(\xi)| &\le C_m \sup_{|\beta| \le m} |\partial^\beta \widehat K(-\xi_{1,1} - B \xi_{2,1}, -C_1 \xi_{1,2} - C_2 \xi_{2,2})|\\
        &\le C_m (|\xi_{1,1} + B\xi_{2,1}| + |C_{1} \xi_{1,2} + C_2 \xi_{2,2}|)^{-|\alpha|}\\
        &\le C (\dist(\xi, \Gamma_s))^{-|\alpha|} \le C (\dist(\xi, \Gamma'))^{-|\alpha|} \,.
    \end{align*}
    Here the second to last line follows from the fact that both $\dist(\xi, \Gamma_s)$ and the expression
    \[
        |\xi_{1,1} + B\xi_{2,1}| + |C_{1} \xi_{1,2} + C_2 \xi_{2,2}|
    \]
    define norms on the finite dimensional quotient space $\Gamma / \Gamma_s$, and are hence comparable. 

    Theorem \ref{thm type 03} now follows from Theorem \ref{thm nondegenerate}.
\end{proof}

\section{Proof of Theorem \ref{thm 3 twisted}}
\label{sec twisted}
For $x\in (\R^n)^6$ we write  $x=(x^0,x^1)$,  where $x^0=(x_1^0,x_2^0,x_3^0)\in (\R^n)^3,\, x^1=(x_1^1,x_2^1,x_3^1)\in (\R^n)^3$, and  we write  $z=(z_1,z_2,z_3)\in (\R^n)^3$. 
 Reading off of Table \ref{table htype}, one finds that a singular Brascamp-Lieb datum  associated with  
\[\mathbf{M}=(J_1^{(1)}\oplus J_1^{(2)} \oplus J_1^{(3)} \oplus C_1)^{\oplus n}\]
is $\mathbf{H}$ with projections $\Pi_1,\Pi_2,\Pi_3:(\R^{n})^9\to (\R^{n})^4$ and $\Pi_0:(\R^{n})^9\to (\R^{n})^5$, given by
\[\Pi_1(x,z) = (x_1^0,x_2^0,x_3^0,z_3)\,,\]
\[\Pi_2(x,z) = (x_1^0+x_1^1, x_2^0+x_2^1, x_3^0,z_2+z_3)\,,\]
\[\Pi_3(x,z) = (x_1^0+x_1^1, x_2^0, x_3^0+x_3^1, z_1+z_3)\,,\]
\[\Pi_0(x,z) = (x_1^1, x_2^1, x_3^1, z_1,z_2)\,.\]
A datum with the projections $\pi_1,\pi_2,\pi_3:(\R^{n})^9\to (\R^{n})^4$ and $\pi_0:(\R^{n})^9\to (\R^{n})^5$, where
\[\pi_1(x,z) = (x_1^1,x_2^0,x_3^0,z_3)\,,\]
\[\pi_2(x,z) = (x_1^0, x_2^1, x_3^0, z_2)\,,\]
\[\pi_3(x,z) = (x_1^0, x_2^0, x_3^1, z_1)\,,\]
\[\pi_0(x,z) = (x^1-x^0, z_1-z_3,z_2-z_3)\,,\]
is equivalent to the datum $\mathbf{H}$. This can be seen by first performing the change of variables $(x^1,z_1,z_2)\to (x^1-x^0, z_1-z_3, z_2-z_3)$ in the singular Brascamp-Lieb form,     relabeling $(x_1^0,x_1^1)$ to $(x^1_1,x^0_1)$, and then replacing $K$ by its reflection in the first fiber $K(-\cdot,\cdot,\cdot, \cdot,\cdot)$. 
Therefore, to prove  Theorem \ref{thm 3 twisted}, it   suffices to show
\begin{equation*}
%    \label{eq:trilin}
    \Big| \int_{(\R^{n})^9} \Big( \prod_{j=1}^3f_j(\pi_j(x,z)) \Big) K(\pi_0(x,z)) \, d(x,z) \Big| \leq C \prod_{j=1}^3 \|f_j\|_{p_j}
\end{equation*}
whenever $2<p_1,p_2,p_3<\infty$, $\frac{1}{p_1}+\frac{1}{p_2}+\frac{1}{p_3}=1$.

We begin with a cone decomposition of   $\widehat{K}$. 
For a  function $\phi$ on $\R^d$, $d\geq 1$, and $t>0$, let \[\phi_t(x) = t^{-d}\phi(t^{-1}x)\,.\]  
Let $B(0,R)$ denote the Euclidean open ball centered at $0$ with radius $R$.   Let $\psi:\R^n\to \R$ be a radial Schwartz function with
 \[\textup{supp}(\widehat\psi)\subseteq B(0,1)\setminus B(0,1/4)\,,\]
satisfying
$ \int_0^\infty \widehat{\psi}(t\xi)\frac{dt}{t}=1$
 for each $\xi\neq 0$. Let $\xi=(\xi_1,\ldots, \xi_5) \in (\R^{n})^5$
and write
 \begin{equation}
     \label{khat}
     \widehat{K}(\xi) = \widehat{K}(\xi) \prod_{j=1}^5 \int_0^\infty \widehat{\psi}(t_j\xi_j)\, \frac{dt_j}{t_j}\,.  
 \end{equation}
We rewrite this as a sum of integrals over five regions, depending on which parameter $t_i$ is the smallest. That is, we write
\[{K} =\sum_{i=1}^5{K_i}\,, \]
where  
\[\widehat{K_i}(\xi) = \int_{T_i}\widehat{K}(\xi) \Big( \prod_{j=1}^5 \widehat{\psi}(t_j\xi_j)t_j^{-1}\Big) \, d(t_1,\ldots, t_5)\,, \]
and
\[ T_i = \{(t_1,t_2,t_3,t_4,t_5)\in (0,\infty)^5: t_i\leq t_j \textup{ for }j\neq i\}\,.\]
We denote 
 \[\widehat{\varphi}(\eta) =  \int_{1}^\infty \widehat{\psi}(s\eta)  \,\frac{ds}{s}\,,\] 
which is
   a radial smooth function supported in 
$B(0,1)$ 
and  for $s_0>0$, it holds
 $\widehat{\varphi_{s_0}}(\eta) =\widehat{\varphi}(s_0\eta) = \int_{s_0}^\infty \widehat{\psi
}(s\eta)  \frac{ds}{s}$\,. 
On each $T_i$ we then integrate  over all larger parameters $t_j$, giving   
 \begin{equation*}
%     \label{khat2}
     \widehat{K_i}(\xi) =   \int_{0}^\infty \widehat{K}(\xi) \widehat{\psi}(t\xi_i)\Big( \prod_{j=1, j\neq i }^5 \widehat{\varphi}(t\xi_j) \Big )\,\frac{dt}{t}\,.
 \end{equation*}
We will proceed with additional decompositions of  the kernels $K_i$. By symmetry in \eqref{khat} and in  the projections $\pi_0, \pi_1,\pi_2,\pi_3$, it suffices to consider $i=1$ and $i=4$ only. 
 
We begin with $i=1$. Here we will decompose the kernel further into Gaussian functions, which will be convenient later in the proof. Let $g$ be the Gaussian $g(x)=e^{-\pi |x|^2}$, whose dimension should always be understood from the context.
  Then we  write 
        \[\widehat{K_1}(\xi) = \int_{0}^\infty  m_1(t\xi, t) |t\xi_1|^2 \widehat{g_{t}}(\xi_1,\xi_2,\xi_3)^2 \widehat{g_t}(\xi_4,\xi_5,\xi_4+\xi_5) \,\frac{dt}{t}\, ,\]
      where  
\[m_1(\xi,t) =    \widehat{K}(t^{-1}\xi)\widehat{\psi}(\xi_1) \Big(\prod_{j=2}^5\widehat{\varphi}(\xi_j) \Big)  ( |\xi_1|^2 g(\xi_1,\xi_2,\xi_3)^2 g(\xi_4,\xi_5,\xi_4+\xi_5))^{-1}\,. \]
Since $|\xi_1|$ is bounded away from zero by the support assumption on $\widehat{\psi}$, the  function $m_1(\xi, t)$ is smooth in the $\xi$ variable.  Moreover,  on the support of $m_1(\xi,t)$,  $\xi \in [-1,1]^{5n}$. 
Denote \[c_{1}(a,t) = (1+|a|)^{6n}  \widecheck{m_1}(a,t)\,,\]
where the inverse Fourier transform is taken in the $\xi$ variable.
A standard integration by parts argument, together with the symbol estimates \eqref{czkernel},  gives 
\begin{equation}
    \label{symest}
    |c_{1}(a,t)|\leq C_0 (1+|a|)^{-16n} 
\end{equation}
for an absolute constant $C_0$, provided $l\geq 22n$.  The lower bound on $l$ is  chosen crudely such that the decay of the coefficients $c_{1}(a,t)$  suffices in all of the arguments below.  We remark that we do not aim to optimize our arguments to minimize this bound.

Taking the inverse Fourier transform of $m$, we write
\begin{equation}
    \label{decomp}
 \widehat{K_1}(\xi)    =  \int_{(\R^{n})^5}  (1+|a|)^{-6n}\int_{0}^\infty c_{1}(a,t)\,  |t\xi_1|^2\widehat{g_{t}}(\xi_1,\xi_2,\xi_3)^2 \widehat{g_t}(\xi_4,\xi_5,\xi_4+\xi_5)e^{-2\pi i a\cdot t\xi}   \,  \frac{dt}{t} \, da \, .
 \end{equation}
 By Fourier inversion, 
\[K_1(\pi_0(x,z)) = \int_{(\R^n)^5} \widehat{K_1}(\xi) e^{2\pi i (\xi\cdot  \pi_0(x,z))}\, d\xi\,. \]
Using the definition of $\pi_0$ and \eqref{decomp}, we can therefore write $K_1(\pi_0(x,z))$
as a    superposition of the kernels of the form
\[\sum_{i=1}^n\int_{0}^\infty  \int_{(\R^{n})^5} c_{1}(a,t)   
\widehat{(\partial_ig)_{t}}(-\xi_1,-\xi_2,-\xi_3)
e^{2\pi i x^0\cdot (-\xi_1,-\xi_2,-\xi_3)}  
 \widehat{(\partial_ig)_{t}}(\xi_1,\xi_2,\xi_3)   \]
 \begin{equation}
     \label{K1dec}
     \times e^{2\pi i (x^1+t(a_1,a_2,a_3))\cdot (\xi_1,\xi_2,\xi_3)} \widehat{g_t}(\xi_4,\xi_5,-\xi_4-\xi_5) e^{2\pi i ((z_1+ta_4)\xi_4 + (z_2+ta_5)\xi_5 - z_3(\xi_4+\xi_5) )}\, 
d\xi \, \frac{dt}{t}\,,
 \end{equation}
weighted by $(1+|a|)^{-6n}$, where $a=(a_1,\ldots, a_5)$.
Here we have also used for convenience that $\partial_ig$ is odd and that $g$ is even,  and replaced $a$ by $-a$.  

Fixing $i$ and $t$, the integral in $\xi$ can be viewed as the integral of the function
\[(\eta_1\ldots,\eta_9) \mapsto  \widehat{(\partial_ig)_{t}}(\eta_1,\eta_2,\eta_3)
e^{2\pi i x^0\cdot (\eta_1,\eta_2,\eta_3)}  
 \widehat{(\partial_ig)_{t}}(\eta_4,\eta_5,\eta_6) e^{2\pi i (x^1+t(a_1,a_2,a_3))\cdot (\eta_4,\eta_5,\eta_6)}  \]
 \[\times \widehat{g_t}(\eta_7,\eta_8,\eta_9) e^{2\pi i (z_1+ta_4, z_2+ta_5, z_3)\cdot (\eta_7,\eta_8,\eta_9)} \]
over the five-dimensional subspace 
\[\{(\eta_1\ldots,\eta_9) \in (\R^n)^9: (\eta_1,\eta_2,\eta_3) = (-\eta_4,-\eta_5,-\eta_6), \eta_9=-\eta_7-\eta_8 \}\,.\]
It equals the integral of the inverse Fourier transform of this function over the orthogonal complement of this subspace, 
\[\{(r_1,\ldots, r_9)\in(\R^n)^9 :  (r_1,r_2,r_3)=(r_4,r_5,r_6), r_7=r_8=r_9\}\,.\]
Therefore,  the term for a fixed $i$ in  \eqref{K1dec} can be written, up to a constant, as 
\begin{equation}
    \label{eq:dec1}
    \int_{0}^\infty  \int_{(\R^{n})^4} c_{1}(a,t)  (\partial_i\partial_{3n+i} g)_{t}((x,z)+r^\star+(0,0,0,ta,0)) \, 
dr \, \frac{dt}{t}\,, 
\end{equation}
where $r=(r_1,r_2,r_3,r_4)$ and $r^\star=(r_1,r_2,r_3, r_1,r_2,r_3, r_4,r_4,r_4)$. 
Thus, it suffices to bound the form in which $K_1(\pi_0(x,z))$ is replaced by  \eqref{eq:dec1}, with estimates uniform  in $a$. Then  it remains to sum over $i$ and     integrate in $a$. By symmetry, it suffices to prove bounds for $i=1$.  This will be done in  Proposition \ref{prop:kernel1}. 

Next we decompose the kernel  $K_4$. We write it as $K_4=K_{6} + K_{7}$, where 
\[\widehat{K_{6}}(\xi) = \int_0^\infty \widehat{K}(\xi) \widehat{\varphi_t}(\xi_1)\widehat{\varphi_t}(\xi_2)\widehat{\varphi_t}(\xi_3)\widehat{\psi_t}(\xi_4)\widehat{{\varphi_{2^{10}t}}}(\xi_5)\, \frac{dt}{t} \,,\]
\[\widehat{K_{7}}(\xi) = \int_0^\infty \widehat{K}(\xi) \widehat{\varphi_t}(\xi_1)\widehat{\varphi_t}(\xi_2)\widehat{\varphi_t}(\xi_3)\widehat{\psi_t}(\xi_4)(\widehat{\varphi_t}-\widehat{{\varphi_{2^{10}t}}})(\xi_5)\, \frac{dt}{t}\,. \]
Note that if $\xi$ is  in the support of the integrand in $\widehat{K_{6}}$ for a fixed $t$, then   $2^{-3}<|t\xi_4+t\xi_5|<2^2$. On the other hand,   if $\xi$ is in the support of the integrand of $\widehat{K_{7}}$ for a fixed $t$,  then $0<|t\xi_4+t\xi_5| < 2^2$, but $|t\xi_5|\geq 2^{-12}$  
We will decompose these multiplier symbols further. To reduce the amount of notation we will use Gaussians for this decomposition as well, even though one could proceed with other Schwartz functions. 

Now we write
\[\widehat{K_{6}}(\xi) = \int_{(\R^n)^5} (1+|a|)^{-6n} \int_0^\infty c_{6}(a,t) |t\xi_4|^2\widehat{g_t}(\xi)|t\xi_4+t\xi_5|^2\widehat{g_t}(\xi_4+\xi_5) e^{-2\pi i a \cdot t\xi} \, \frac{dt}{t}\,  da\,,\]
where $c_{6}(a,t) =   (1+|a|)^{6n} \widecheck{m_{6}}(a,t)$ and 
\[m_{6}(\xi,t) =   \widehat{K}(t^{-1}\xi) \widehat{\varphi}(\xi_1)\widehat{\varphi}(\xi_2)\widehat{\varphi}(\xi_3)\widehat{\psi}(\xi_4)\widehat{{\varphi_{2^{10}}}}(\xi_5)   (|t\xi_4|^2\widehat{g_t}(\xi)|t\xi_4+t\xi_5|^2\widehat{g_t}(\xi_4+\xi_5))^{-1}\,. \]
Here, $c_{6}(a,t)$ satisfy the symbol estimate analogous to \eqref{symest}. 

Thus,   $K_{6}(\pi_0(x,z))$ can be written 
as a weighted  superposition of integrals of the form
\[\int_{0}^\infty  \int_{(\R^n)^5} c_{6}(a,t)   
\widehat{g_t}(\xi_1,\xi_2,\xi_3)  \widehat{(\Delta g)_t}(\xi_4)  \widehat{g_t}(\xi_5) \widehat{(\Delta g)_t}(\xi_4+\xi_5)  \]
\[\times e^{2\pi i (x^1-x^0,z_1-z_3,z_2-z_3)\cdot \xi}  e^{2\pi i ta\cdot \xi}  
\, d\xi \, \frac{dt}{t}\]
\[= \int_{0}^\infty  \int_{(\R^n)^5} c_{6}(a,t)   \widehat{g_t}(\xi_1,\xi_2,\xi_3)
 \widehat{(\Delta g)_t}(\xi_4) \widehat{g_t}(\xi_5) \widehat{(\Delta g)_t}(\xi_4+\xi_5)
e^{2\pi i (x^1-x^0+t(a_1,a_2,a_3))\cdot (\xi_1,\xi_2,\xi_3)}\]
\begin{equation}\label{eq:decK6}
\times e^{2\pi i (z_1+ta_4,z_2+ta_5,z_3)\cdot (\xi_4,\xi_5,-\xi_4-\xi_5)} 
\, d\xi \, \frac{dt}{t}\,.  
\end{equation}
Then, the integral in $\xi_4,\xi_5$ can be seen as the integral 
of  the function 
\[(\eta_1,\eta_2,\eta_3) \mapsto    \widehat{(\Delta g)_t}(\eta_1)\widehat{g_t}(\eta_2)  \widehat{(\Delta g)_t}(\eta_3) e^{2\pi i (z_1+ta_4,z_2+ta_5,z_3)\cdot (\eta_1,\eta_2,\eta_3)}  \]
 over the subspace \[\{(\eta_1,\eta_2,\eta_3) \in (\R^n)^3: \eta_3=-\eta_1-\eta_2\}\,.\] 
It equals the  integral of the Fourier transform of this function over the orthogonal complement
  \[\{(r_1,r_2,r_3) \in (\R^n)^3 : r_1=r_2=r_3\}\,.\] Using this and taking     the inverse Fourier transform in $\xi_1,\xi_2,\xi_3$,  the   display \eqref{eq:decK6} is up to a constant equal to 
\begin{equation*}
\int_{0}^\infty  \int_{\R^n} c_{6}(a,t)   
g_t(x^1-x^0+t(a_1,a_2,a_3))(\Delta g)_t(z_1+r+ta_4)g_t(z_2+r+ta_5)(\Delta g)_t(z_3+r)\, 
dr \, \frac{dt}{t} \,.
\end{equation*}
 Thus, it suffices to bound a form with $K_{6}(\pi_0(x,z))$ replaced by this kernel, with a constant uniform in $a$. This will follow from Proposition \ref{prop:kernel2}.

For the kernel $K_{7}$ we proceed with a similar decomposition  but with a factor $|t\xi_5|^2$ instead of $|t\xi_4+t\xi_5|^2$. This leads to bounding a form with $K_{7}(\pi_0(x,z)))$ replaced by
\begin{equation*}
    \int_{0}^\infty  \int_{\R^n} c_{7}(a,t)   
g_t(x^1-x^0+t(a_1,a_2,a_3))(\Delta g)_t(z_1+r+ta_4)(\Delta g)_t(z_2+r+ta_5)g_t(z_3+r)\, 
dr\,  \frac{dt}{t}\,.
\end{equation*}
with a constant uniform in $a$, where $c_{7}(a,t)$ satisfies a  bound analogous to  \eqref{symest}. 
Note a symmetry between the last two displays, which can be seen by interchanging $z_2$ and $z_3$, translating $r\to r-ta_5$, and replacing  $a_4-a_5$ by $a_4$ in the second display. 
Bounds  for this form   will also follow from Proposition \ref{prop:kernel2}.  

 To summarize, we have reduced Theorem \ref{thm 3 twisted} to the following two propositions. 
  \begin{proposition}
     \label{prop:kernel1} Let $n\geq 1$. 
     Let $2<p_1,p_2,p_3<\infty$ and  $\frac{1}{p_1}+\frac{1}{p_2}+\frac{1}{p_3}=1$. 
     There exists a constant $C>0$ such that for each   $a\in (\R^n)^9$,  $c(t)$ satisfying  $|c(t)|\leq 1$ for each $t>0$, and all Schwartz functions $f_1,f_2,f_3:(\R^n)^4\to \mathbb{C}$,
     \begin{equation*}
   \Big|  \int_{0}^\infty c(t) \int_{(\R^n)^{13}} \Big( \prod_{j=1}^3f_j(\pi_j(x,z)) \Big)      
 (\partial_1\partial_{3n+1} g)_{t}((x,z)+r^\star+ta) \, d(x,z,r)\, \frac{dt}{t}\Big | 
 \end{equation*}
 \begin{equation*}
     \leq C (1+|a|)^{8n} \prod_{j=1}^3 \|f_j\|_{p_j}\,,
 \end{equation*}
where $r=(r_1,r_2,r_3,r_4)$,  $r^\star=(r_1,r_2,r_3, r_1,r_2,r_3, r_4,r_4,r_4)$.
 \end{proposition}
 
 \begin{proposition}
 \label{prop:kernel2} Let $n\geq 1$. 
       Let $1<p_1,p_2,p_3<\infty$ and $\frac{1}{p_1}+\frac{1}{p_2}+\frac{1}{p_3}=1$. 
     There exists a constant $C>0$ such that for each   $a\in (\R^n)^6$,  $c(t)$ satisfying  $|c(t)|\leq 1$ for each $t>0$, and all Schwartz functions $f_1,f_2,f_3:(\R^n)^4\to \mathbb{C}$,
     \begin{equation*}
        \Big|  \int_{0}^\infty c(t) \int_{(\R^n)^{10}} \Big( \prod_{j=1}^3f_j(\pi_j(x,z)) \Big)    g_t(x^1-x^0+t(a_1,a_2,a_3))    {(\Delta g)_t}(z_1+r+ta_4)
     \end{equation*}
     \[\times g_t(z_2+r+ta_5)(\Delta g)_t(z_3+r+ta_6) \, d(x,z,r) \,\frac{dt}{t}\Big| \leq C (1+|a|)^{16n}\prod_{j=1}^3 \|f_j\|_{p_j}\,.\]
 \end{proposition}
Note that in Proposition \ref{prop:kernel2}, one could decompose the Gaussian in $x^1-x^0$   into another integral of translated Gaussians. Also, the Laplacians could be split  into sums of second order derivatives. This would yield   a form that is, schematically,  similar to the one   in Proposition \ref{prop:kernel1}. However, we chose not to do this  as, it will not be needed for the proof. 
 
Proposition \ref{prop:kernel1} will be proven using  twisted techniques, while Proposition \ref{prop:kernel2} will follow from the classical square and maximal function bounds, similarly as in the case of the Coifman-Meyer multipliers.  
We will prove these   propositions in the following two sections.

\subsection{Proof of Proposition \ref{prop:kernel2}} Denote  $h=\Delta g$. 
Using the definition of the  projections $\pi_j$ and splitting the Gaussian into tensor products of three lower-dimensional Gaussians, the form we need to bound reads 
\[ \int_{0}^\infty c(t) \int_{(\R^n)^{10}} f_1(x_1^1,x_2^0,x_3^0,z_3)f_2(x_1^0, x_2^1, x_3^0, z_2) f_3(x_1^0, x_2^0, x_3^1, z_1)\]
\[
\times\Big( \prod_{j=1}^3g_t(x_j^1-x_j^0+ta_j) \Big ){h_t}(z_1+r+ta_4)g_t(z_2+r+ta_5)h_t(z_3+r+a_6) \, d(x,z,r) \, \frac{dt}{t} \,.\]
Integrating in $x_1$ and $z$, using  that Gaussians are even and replacing $a_1,a_2,a_3$ by $-a_1$, $-a_2$, $-a_3$, it suffices to estimate
\[ \Big| \int_{0}^\infty c(t) \int_{(\R^n)^4} (f_1*_{1,4}(g_{t,a_1}\otimes h_{t,a_6}))(x^0,r) (f_2*_{2,4}(g_{t,a_2}\otimes g_{t,a_5}))(x^0, r)\]
\[\times (f_3*_{3,4}(g_{t,a_3}\otimes h_{t,a_4}))(x^0, r)\, d(x^0,r) \, \frac{dt}{t} \Big|\,,\]
where the subscript $*_{m_1,m_2}$ means that we take $2n$-dimensional convolutions with the functions $f_j$ in the coordinates $m_1n,\ldots, m_1(n+1)$ and $m_2n,\ldots, m_2(n+1)$. Here,    all functions $g$ and $h$ that appear in the tensor products are $n$-dimensional,  we have denoted \[g_{t,a_j} = t^{-n}g_t(\cdot +ta_j)\,,\]
and analogously for $h_{t, a_j}$. For two functions $\phi,\rho$ we also write $(\phi\otimes \rho)(u,v)=\phi(u)\rho(v)$.
Applying H\"older's inequality in $t$ for the exponents $(2,\infty,2)$ and  using   $|c(t)|\leq 1$, we bound the last display by 
\[ \int_{(\R^n)^4}\Big(\int_0^\infty |(f_1*_{1,4}(g_{t,a_1}\otimes h_{t,a_6}))(x^0,r)|^2 \frac{dt}{t}\Big)^{1/2} \sup_{t>0}|(f_2*_{2,4}(g_{t,a_2}\otimes g_{t,a_5}))(x^0, r)| \]
\[\times \Big(\int_0^\infty |(f_3*_{3,4}(g_{t,a_3}\otimes h_{t,a_4}))(x^0, r)|^2 \frac{dt}{t}\Big)^{1/2} d(x^0,r)\,.\]
Applying H\"older's inequality in $(x^0,p)$, we estimate this further by   
\[  \Big\| \Big(\int_0^\infty |f_1*_{1,4}(g_{t,a_1}\otimes h_{t,a_6})|^2 \frac{dt}{t}\Big)^{1/2} \Big \|_{p_1} \Big \|\sup_{t>0}|f_2*_{2,4}(g_{t,a_2}\otimes g_{t,a_5})|\Big \|_{p_2}\]
\[\times  \Big\| \Big(\int_0^\infty  |f_3*_{3,4}(g_{t,a_3}\otimes h_{t,a_4})|^2 \frac{dt}{t}\Big)^{1/2}\Big\|_{p_3}  .\]
Using   bounds for the two-dimensional fiber-wise maximal and square functions, the last display is bounded by an absolute constant times 
\[ (1 + |a|)^{6n+2}\|f_1\|_{p_1}\|f_2\|_{p_2}\|f_3\|_{p_3}\,.\]
For the maximal function, a polynomial loss in the shift follows by dominating the Gaussian $g_{t,a_j}\leq 10(2(1 + |a_j|))^n g_{2t(1+|a_j|)}$. 
For the shifted square function, we have a uniform bound on $L^2$ in the shift $a$. We also have a weak $L^1$ bound  with polynomial loss $(1 + |a_j|)^{n+1}$, this follows from the standard proof of weak $L^1$ bounds of Calderón-Zygmund operators, but with an $L^2(\frac{dt}{t})$-vector valued kernel, as in \cite[Section 5.6.1]{grafakos2008classical}. Interpolation and duality then give the polynomial loss $(1+ |a_j|)^{n+1}$.

\subsection{Proof of Proposition \ref{prop:kernel1}} %\label{sec:twisted}
In contrast with the previous section, now we cannot bound the form by the maximal and square functions of each of the functions $f_j$ separately.  

To prove Proposition \ref{prop:kernel1},  we will first prove an estimate for a local version of our form. Local estimates  will then be combined into a global estimate using a stopping-time argument. 
 A finite collection $\T$ of dyadic cubes in $\R^{d}$, $d\geq 1$, is called a convex tree if there exists $Q_{\T}\in \T$ such that $Q\subseteq Q_{\T}$ for every $Q\in \T$ and if $Q,Q''\in \T$ and $Q\subseteq Q' \subseteq Q''$, then $Q'\in \T$. If $\ell(Q)$ denotes the side-length of a dyadic cube $Q$, we denote
\[\Omega_{\T} = \cup_{Q\in \T} Q\times (\ell(Q)/2,\ell(Q))\,.\]
For $f\in L^2_{\textup{loc}}(\R^d)$ we also define a variant of a maximal operator on a tree $\T$
\[Mf(\T) = \sup_{Q\in \T}\sup_{Q'\supseteq Q } \Big(\frac{1}{|Q'|} \int_{Q'} |f|^2 \Big)^{1/2}, \]
where the second supremum is over all cubes $Q'$ with sides parallel to the coordinate axes, which contain  the   cube $Q$.

Let $c(t)$ and $\pi_1,\pi_2,\pi_3$  be as in Proposition \ref{prop:kernel1}. Let 
$\pi_4:(\R^{n})^9\to (\R^{n})^4$ be  given by
\[\pi_4(x,z) = (x_1^1,x_2^1,x_3^1,z_2)\,.\]
We   will prove bounds for a more symmetric  local quadrilinear form 
\[\Lambda_{\T,a}(f_1,f_2,f_3,f_4)= (1+|a|)^{-16n}
  \int_{\Omega_{\T}}  {c(t)}  \int_{(\mathbb{R}^{n})^9}\Big( \prod_{j=1}^4f_j(\pi_j(x,z)) \Big ) \]
  \[\times  (\partial_1\partial_{3n+1} g)_{t}((x,z)+r^\star+ta)\, 
d(x,z)\, dr\, \frac{dt}{t}\, , \]
defined for bounded functions $f_j$ on $\R^{4n}$   and a convex tree $\T$ in $\R^{4n}$.
The main step in the proof of Proposition \ref{prop:kernel1} is the following estimate, which   will be  applied  with $f_4=1$.
\begin{proposition}\label{prop:localabc}
 There exists a constant $C>0$ such that for any convex tree $\T$, any $a \in (\R^n)^9$, and   bounded  functions  $f_1,f_2,f_3,f_4: (\R^n)^4 \to \mathbb{C}$, 
 \[|\Lambda_{\T,a}(f_1,f_2,f_3,f_4)| \leq C |Q_{\T}| \prod_{i=1}^4    Mf_i(\T) \,. \]
\end{proposition}

\begin{proof}
[Proof of Proposition \ref{prop:localabc}]
We may assume that the functions $f_j$ are real-valued, as otherwise we split them into real and imaginary parts. Interchanging the order of integration, using   the triangle inequality,  and   $|c(t)|\leq 1$,  we bound   $|\Lambda_{\T,a}(f_1,f_2,f_3,f_4)|$ by 
\[ (1+|a|)^{-16n} \int_{\Omega_\T} \int_{(\R^n)^{7}} \Big | \int_{\R^n}f_1(x^1_1,x^0_2,x^0_3,z_3)f_4(x_1^1,x_2^1,x_3^1,z_2)(\partial_1g)_t(x^1_1+r_1+a_1t) \, dx^1_1 \Big | \]
\[\times \Big| \int_{\R^n}f_2(x^0_1, x^1_2, x^0_3,z_2) f_3(x^0_1, x^0_2, x^1_3,z_1) (\partial_1g)_t(x^0_1+r_1+a_4 t)\, dx^0_1 \Big| \, d\mu \, ,\]
where
\[d\mu =   g_{t}((x^0_2, x^0_3, x^1_2, x^1_3)+(r_2,r_3,r_2,r_3)+t(a_2, a_3, a_5,a_6), z+(r_4,r_4,r_4)+t(a_7,a_8,a_9)) \]
\[\times d(x^0_2, x^0_3, x^1_2, x^1_3,z)\, dr\, \frac{dt}{t} \, . \]
Applying the Cauchy-Schwarz inequality with respect to $d\mu$ bounds this form by $(1+|a|)^{-16n}$ times   the geometric mean of 
\begin{equation} \label{aftercs1}
     \int_{\Omega_\T} \int_{(\R^n)^7} \Big| \int_{\R^n}f_1(x^1_1,x^0_2,x^0_3,z_3)f_4(x_1^1,x_2^1,x_3^1,z_2)(\partial_1g)_t(x^1_1+r_1+a_1t) dx^1_1 \Big |^2 d\mu\, .
\end{equation}
and 
\[   \int_{\Omega_\T} \int_{(\R^n)^{7}} \Big | \int_{\R^n}f_2(x^0_1, x^1_2, x^0_3,z_2) f_3(x^0_1, x^0_2, x^1_3,z_1) (\partial_1g)_t(x^0_1+r_1+a_4 t) dx^0_1  \Big |^2 d\mu \,.\]
 These two terms are analogous, which can be seen by swapping the roles of $x_3^0$ and $x_3^1$ in the second term. Thus, it suffices to proceed with \eqref{aftercs1}.

We integrate in $z_1$ and then expand out the square in \eqref{aftercs1}. This gives 
\[     \int_{\Omega_\T} \int_{(\R^n)^8}  f_1(x^0_1,x^0_2,x^0_3,z_3)f_4(x_1^0,x_2^1,x_3^1,z_2) f_1(x^1_1,x^0_2,x^0_3,z_3)f_4(x_1^1,x_2^1,x_3^1,z_2)  \]
\begin{equation}
    \label{d:aftercs}
    \times (\partial_1g)_t((x^0,z_2)+r+u_1t)(\partial_1g)_t((x^1,z_3)+r+u_2t)d(x,z_2,z_3) \, dr \, \frac{dt}{t} \,,
\end{equation}
where $u_1=(a_1,a_2,a_3,a_8),\,u_2=(a_1,a_5,a_6,a_9)$.  

If $n=1$, this expression corresponds to 
 the   local form $\widetilde{\Lambda}_{\T}$  from the dimension four case
 in \cite{DST22}. More precisely, it can be interpreted as the local form applied to a  $16$-tuple of  functions on $\R^4$, in the case of the identity matrix, and when all but four functions are set to the constant  $1$. 
 Applying  the result from \cite{DST22}  to this setup yields a variant of  Proposition \ref{prop:localabc},     where the  maximal operators  $Mf_i$ are replaced by $(M|f_i|^4)^{1/4}$. However, this is insufficient to establish Proposition \ref{prop:kernel1}. It is therefore  essential to view \eqref{d:aftercs} as a variant of the quadrilinear form from the two-dimensional case in \cite{DST22}, but acting on functions defined on $\R^n\times \R^{3n}$ instead of $\R\times \R$.  The variables are now
 \[x^0_1,x_1^1\in \R^n,\quad  (x^0_2,x^0_3,z_3), (x^1_2,x^1_3,z_2)\in \R^{3n}\, .\] This perspective leads to only one more application of the  Cauchy-Schwarz inequality, which subsequently gives the maximal operators $Mf_i$ and the desired H\"older estimate. 

The paper \cite{DST22} establishes an estimate for this quadrilinear form in the setting of functions on $\R\times \R$.   The argument, however, extends to $\R^n\times \R^{3n}$ without   significant complications. Since   \cite{DST22}    additionally focuses on other matters, we nevertheless prove the estimate in our specific setting for the reader's convenience and to maintain a self-contained exposition.

We will rewrite the form \eqref{d:aftercs} more concisely, for which we introduce slightly more general expressions. 
Let $d_1,d_2\geq 1$.
 For $y\in (\R^{d_1}\times \R^{d_2})^2$  we write $y=(y^0,y^1)$, where $y^0=(y_1^0,y_2^0)\in \R^{d_1}\times \R^{d_2},\, y^1=(y_1^1,y_2^1)\in \R^{d_1}\times \R^{d_2}$, and let $q=(q_1,q_2)\in \R^{d_1}\times \R^{d_2}$.  
 We   write $m=d_1+d_2$ and identify $\R^{d_1}\times \R^{d_2}$ with $\R^m$.
Define the maps $\rho_j: (\R^m)^2 \to \R^m$ by   
\[\rho_1(y)=(y_1^0,y_2^0)\,,\quad \rho_3(y)=(y_1^0,y_2^1)\,,\]
\[\rho_2(y)=(y_1^1,y_2^0)\,,\quad \rho_4(y)=(y_1^1,y_2^1)\,.\]
For  bounded functions $F_1,\ldots,  F_4$ on $\R^m$,  $v\in (\R^m)^2$, and $\T$ a convex tree in $\R^m$, 
we define  
\[\Theta_{\T,v}(F_1,F_2,F_3,F_4) =  (1+|v|)^{-4m}   \int_{\Omega_\T} \int_{(\R^{m})^2}  \Big( \prod_{j=1}^4 F_j(\rho_j(y)) \Big) \]\[\times (\partial_1\partial_{m+1}g)_{t}(y+(q,q)+vt) \,  dy \, dq \, \frac{dt}{t} \,.\]
Then,    \eqref{d:aftercs} can be recognized as
\[(1+|u|)^{4m}\Theta_{\T,u}(f_1,f_1,f_4,f_4) \]
with $u=(u_1,u_2)$, $d_1=n$, $d_2=3n$.
Recall that $u$ consists of the components of $a$ and satisfies $|u|\leq C|a|$. Thus, it  will  suffice to prove 
\begin{equation}
    \label{est:thetabd}
    |\Theta_{\T,u}(f_1,f_1,f_4,f_4)|\leq C   |Q_\T|Mf_1(\T)^2Mf_4(\T)^2 \,.
\end{equation}

To show this estimate, we will remove the localization of the kernel and localize the functions, similarly as in \cite{DST22}.  This will allow for translation $q\to q-vt$  and global telescoping arguments.
For a tree $\T$ in $\R^{m}$ we   define  a region in $\R^m$ 
\[T_k=\cup\{Q\in \T: \ell(Q)=2^k\}.\]
 For $F_j,\T,v$   as above,  $\alpha\geq 1$, and $1\leq i \leq m$,  we define
\[{\Theta}^{(i)}_{\T,v,\alpha}(F_1,F_2,F_3,F_4)  = (\alpha+|v|)^{-4m}    \sum_{k\in\Z} \int_{2^{k-1}}^{2^k} \int_{\R^m} \int_{(\R^m)^2}  \Big(\prod_{j=1}^4 (F_j1_{T_k})(\rho_j (y))\Big)\]
\[\times (\partial_{i}\partial_{m+i}g)_{Dt}(y+(q,q)+vt)  \,  dy\, dq\, \frac{dt}{t}\, ,  \]
 where   
 $D=D(\alpha)$ is a $2m\times 2m$ diagonal matrix with diagonal entries $d_{ll}=1$ if $1\leq l \leq d_1$ or $d_1+d_2< l \leq 2d_1+d_2$, and $d_{ll}=\alpha$ otherwise. Here, $g_{D t} = (\det D)^{-1}g_{t}(D^{-1}\cdot )$. We will only use this expression   when either  $\alpha=1$ or $v=0$. 
 
The following lemma will reduce the problem to proving a bound for $\Theta_{\T,u,1}^{(1)}$ instead. 
\begin{lemma}\label{lem:swaploc} There exists a constant $C>0$ such that for any   $F_j,\T,v$ as above, 
  \[  |\Theta_{\T,v}(F_1,F_2,F_3,F_4) - {\Theta}^{(1)}_{\T,v,1}(F_1,F_2,F_3,F_4)| \leq C   |Q_\mathcal{\T}| \prod_{i=1}^4   MF_i(\T)\, .\]
\end{lemma}

We will use another lemma, which   will reduce the problem of bounding $\Theta_{\T,0,\alpha}^{(1)}$ to bounding the sum of $\Theta_{\T,0,\alpha}^{(i)}$, $i\neq 1$, instead.
\begin{lemma}\label{lem:tel}  There exists a constant $C>0$ such that for any   $F_j,\T,\alpha$ be as above, 
   \[ \Big | \sum_{i=1}^m\Theta^{(i)}_{\T,0,\alpha}(F_1,F_2,F_3,F_4) \Big  | \leq C  |{Q}_{\T}| \prod_{i=1}^4MF_i(\T)\, .\]
\end{lemma}
We postpone the proofs of these two lemmas  until the end of this section and return to proving \eqref{est:thetabd}.

By Lemma \ref{lem:swaploc},  it thus suffices to show  
\begin{equation*}
    |{\Theta}^{(1)}_{\T,u,1}(f_1,f_1,f_4,f_4)| \leq C  |Q_\mathcal{\T}|   Mf_1(\T)^2Mf_4(\T)^2 \, .
\end{equation*}
Note that we can write $ {\Theta}^{(1)}_{\T,u,1}(f_1,f_1,f_4,f_4)$
in an analogous way as in  \eqref{aftercs1}, by writing the product of all non-positive terms as a square.  Indeed, we can write is as  
\[(1+|u|)^{-4m}   \sum_{k\in\Z} \int_{2^{k-1}}^{2^k}  \int_{\R^{d_1+3d_2}}  \Big|\int_{\R^{d_1}} (f_11_{T_k})(\rho_1 (y))(f_41_{T_k})(\rho_3(y)) (\partial_1g)_t(y_1^0+q_1+a_1t)\, dy_1^0\Big|^2\]
\[\times g_{t}((y_2^0,y_2^1)+(q_2,q_2)+
(a_2,a_3,a_8,a_5,a_6,a_9)t)   \, d(y_2^0,y_2^1,q)\, \frac{dt}{t}\, ,  \]
where   we have unravelled the definition of $u$ inside the Gaussian. 
 We estimate a non-centered Gaussian by a centered Gaussian as 
\[g_t(\cdot + vt)\leq 10 (2(1 + |v|))^{2d_2} g_{2t(1+|v|)}\, ,\]
and apply this to the Gaussian outside of the squared term. 
 We also change variables $q_1\to q_1-u_1t$. This gives a constant multiple of the form  
\[\alpha^{2d_2}(1+|u|)^{-4m}{\Theta}^{(1)}_{\T,0,\alpha}(f_1,f_1,f_4,f_4) \]
  with $\alpha=2(1+|(a_2,a_3,a_8,a_5,a_6,a_9)|)$.  
Since $\alpha^{2d_2}(1+|u|)^{-4m}\leq C$, it will suffice to show
\[{\Theta}^{(1)}_{\T,0,\alpha}(f_1,f_1,f_4,f_4)\leq C  |Q_\mathcal{\T}|   Mf_1(\T)^2Mf_4(\T)^2 \, . \]
Note that by symmetry,  ${\Theta}^{(i)}_{\T,0,\alpha}(f_1,f_1,f_4,f_4)\geq 0$ for each $1\leq i \leq d_1$. 
By Lemma \ref{lem:tel},  it will    thus suffice to prove  
\[\Big |\sum_{i=d_1+1}^m{\Theta}^{(i)}_{\T,0,\alpha}(f_1,f_1,f_4,f_4) \Big |\leq C  |Q_\T|Mf_1(\T)^2Mf_4(\T)^2\, .\]

To show this inequality,  we   bound the left-hand side of the last display by
\[\sum_{i=d_1+1}^{m} \sum_{k\in\Z} \int_{2^{k-1}}^{2^k} \int_{\R^{3d_1+d_2}} \Big | \int_{\R^{d_2}}(f_{1}1_{T_k})(y_1^0,y_2^0)(f_{1}1_{T_k})(y_1^1,y_2^0) (\partial_{i}g)_{\alpha t}(y_2^0+q_2) \, dy^0_2 \Big | \]
\[\times \Big| \int_{\R^{d_2}}(f_41_{T_k})(y_1^0,y_2^1) (f_41_{T_k})(y_1^1,y_2^1) (\partial_{i}g)_{\alpha t}(y_2^1+q_2) dy^1_2 \Big| g_{t}((y_1^0,y_1^1)+(q_1,q_1))   \, d(y_1^0,y_1^1,q) \, \frac{dt}{t} \, .\] 
We apply  the Cauchy-Schwarz inequality in $y_1^0,y_1^1,q,t$ and in the sums, and then expand out the square, similarly as we did in \eqref{d:aftercs}. This gives an estimate by
\[ \prod_{j\in\{1,4\}}\Big(\sum_{i=d_1+1}^m\Theta^{(i)}_{\T,0,\alpha}(f_j,f_j,f_j,f_j)\Big) ^{1/2}. \]
For each $1\leq i \leq m$,  we have  $\Theta^{(i)}_{\T,0,\alpha}(f_j,f_j,f_j,f_j)\geq 0$.   Therefore,   Lemma \ref{lem:tel}   gives 
\[ \Theta^{(i)}_{\T,0,\alpha}(f_j,f_j,f_j,f_j) \leq C   |Q_{\T}| Mf_j(\T)^4 \] 
for each $1\leq i\leq m$ and $1\leq j\leq 4$.  This finishes the proof of Proposition \ref{prop:localabc}, up to verification of Lemmas \ref{lem:swaploc} and \ref{lem:tel}.
\end{proof}

To prove Lemmas \ref{lem:swaploc} and \ref{lem:tel}  we will need the following Brascamp-Lieb inequality. 

\begin{lemma}\label{lem:bl}
    For any   measurable functions $F_1,F_2,F_3,F_4: \R^{m}\to \mathbb{C}$,
    \[\Big|\int_{(\R^m)^2} \Big( \prod_{j=1}^4 F_j(\rho_j(y)) \Big)   dy\Big|\leq \prod_{j=1}^4\|F_j\|_2\, .\]
\end{lemma}

This lemma was proven in the case $d=1$ by repeated applications of the Cauchy-Schwarz inequality in {\cite[Lemma 3.2]{DST22}}. The proof when the variables are in higher dimensions follows in the analogous way and we omit it.

We will also need an estimate on the boundary   of a convex tree from \cite{DST22}.
\begin{lemma}[{\cite[Lemma 4.1]{DST22}}]
\label{lem:bdrytree}
    There exists a constant $C>0$ such that for any convex tree $\T$ in $\R^m$, 
    \[\sum_{k\in\Z} 2^{mk}\# (\partial T_k \cap (2^k\Z)^m)\leq C|Q_{\T}|\, .\]
\end{lemma}

Now we are ready to prove Lemmas \ref{lem:swaploc} and \ref{lem:tel}.
\begin{proof}[Proof of  Lemma \ref{lem:swaploc}] 
We proceed along the lines of the argument  in {\cite[Section 4]{DST22}}  in the case of the identity matrix.
First we use    {\cite[Lemma 3.4]{DST22}}, which in the particular case of the identity matrix gives
%says that if $|q-q'|\leq 2^{m+3}t$,  then  
   \[ \left | {(1+|v|)^{-4m}} (\partial_1\partial_{m+1} g)_{t}(y+(q,q)+vt) \right  | \leq C t^{-2m}\sum_{n\geq 0} 2^{-4nm}\chi(|y+(q,q)|\leq 2^nt) \, ,\]
where $\chi(\mathcal{A})$ equals $1$ if  the condition $\mathcal{A}$ is satisfied and $0$ otherwise. 
%Note that in ${\Theta}_{\T,\alpha}$ we may replace $q$ by $q+bt$ due to the translation invariance. 
With this and   \eqref{symest}, we estimate \[ |\Theta_{\T,v}(F_1,F_2,F_3,F_4) -{\Theta}_{\T,v,1}^{(1)}(F_1,F_2,F_3,F_4)|\] by an absolute constant  times
\begin{equation}
    \label{eq:diff1}
\sum_{n\geq 0} 2^{-4mn} \sum_{k\in \Z} \int_{2^{k-1}}^{2^k} \int_{T_k}\int_{S_k^c} \Big( \prod_{j=1}^4 |F_j(\rho_j(y))| \Big )t^{-2m}  \chi(|y-(q,q)|\leq 2^nt) \, dy\, dq\, \frac{dt}{t}
\end{equation}
\begin{equation}
    \label{eq:diff2}
+ \sum_{n\geq 0} 2^{-4mn} \sum_{k\in \Z} \int_{2^{k-1}}^{2^k} \int_{T_k^c}\int_{S_k} \Big( \prod_{j=1}^4 |F_j(\rho_j(y))| \Big )t^{-2m}  \chi(|y-(q,q)|\leq 2^nt) \, dy\, dq\, \frac{dt}{t}\, ,
\end{equation}
where $S_k=\{y\in \R^{2m}: \rho_j(y)\in T_k\, \textup{ for all }j=1,2,3,4\}$.

First we estimate the summand in  \eqref{eq:diff1} for fixed $n$ and $k$ by 
 \[C 2^{-2mk}\int_{T_k}\int_{S_k^c} \Big( \prod_{j=1}^4 |F_j(\rho_j(y))|  \chi(|\rho_j(y)-q|\leq 2^{n+k}) \Big ) \, dy \, dq\, ,\]
where we used $\rho_j(q,q)=q$.
Let $E$ be the set of $q\in T_k$ such that the inner integral of the last display does not vanish. We estimate the last display using Lemma~\ref{lem:bl} by
\begin{equation}
    \label{eqn:error1}
    C 2^{-2mk} \int_{ E}  \prod_{j=1}^4 \|F_j(w) 
 \chi(|w-q|\le 2^{n+k})\|_{L^2(w)}
\,dq \le C 2^{2mn} |E| \prod_{j=1}^4 
{M}F_j(\T)\, .
\end{equation}

We proceed by estimating $|E|$. If $q\in E$, then there is $y\in S_k^c$ such that for all $j$, 
\[|\rho_j (y)-q|\le 2^{n+k}\, .\]
By definition of $S_k$,
$\rho_{j_0} (y)\not \in  T_k$ for some $1\leq j_0\leq 4$.
Let $Q_q$ be a dyadic cube of side length  $2^k$ containing $q$ and let $Q_y$ be a dyadic cube of side length $2^k$ such that
$\rho_{j_0} (y) \in  Q_y$.
Then $Q_q\subseteq T_k$ and $Q_y \not \subseteq T_k$.
But both $Q_q$ and $Q_y$ are contained in the ball $B$ of radius
$C2^{n+k}$ about $q$ for 
sufficiently large $C>1$.
Therefore, there is
$w\in \partial T_k \cap (2^k \Z)^m$
such that $w\in  B$.
 But then
$q$ is contained in the ball of radius $C2^{n+k}$ about $w$.
 This implies
\[|E|\le C2^{nm+km}  \#(\partial T_k \cap (2^k \Z)^m)\, .\]
Applying this estimate to  
\eqref{eqn:error1} and summing   in $n$ and $k$, we obtain
for \eqref{eq:diff1} a upper bound by a constant times
\[ \Big( \sum_{k \in \Z}2^{km} 
\#(\partial T_k \cap (2^k \Z)^m) \Big )\prod_{j=1}^4 
{M}F_j(\T)\,.\]
 Lemma \ref{lem:bdrytree} then yields the desired bound for \eqref{eq:diff1}.

It remains to estimate \eqref{eq:diff2}.
Fix    $n$ and $k$ and estimate the corresponding summand by  
\[ C 2^{-2mk} \int_{T_k^c}\int_{S_k} \Big( \prod_{j=1}^4 |F_j(\rho_j (y))| 
 \chi(|\rho_j (y)- q|\le 2^{n+k})\Big)
\, dy \, dq\, .\]
Let $E$ be the set of $q\in T_k^c$ such that the inner integral of the last display is not zero. We estimate the last display with Lemma \ref{lem:bl} by
\begin{equation}
    \label{eqn:error2}
     C 2^{-2mk} \int_{ E}  \prod_{j=1}^4 \|F_j(w) 
 \chi(|w-q|\le 2^{n+k})\|_{L^2(w)}
\,dq \,.
\end{equation}
If $p\in E$, then there is $y\in S_k$ such that for all $j$,  
$|\rho_j(y)-q|\le 2^{n+k}.$
By definition of $S_k$, for every $j$ there is $q^{(j)}\in T_k$ such that
$\rho_j(y)=q^{(j)}$.
Using the triangle inequality, we   estimate  \eqref{eqn:error2} by
\begin{equation*}
 C 2^{-2mk} \int_{ E}  \prod_{j=1}^4 \|F_j(w) 
 \chi(|w-q^{(j)}|\le 2^{n+k+1})\|_{L^2(w)}
\,dq \le C 2^{2mn} |E| \prod_{j=1}^4 
{M}F_j(\T)\, .
\end{equation*}
To obtain the last inequality we may argue as for \eqref{eqn:error1}, because $q^{(j)}\in T_k$.
Similarly as before, 
the ball of radius $C2^{n+k+1}$ about $p$ contains $p_j$ and   as before we see that it also contains a point 
in $\partial T_k \cap (2^k \Z)^m$.
We   estimate  
\[|E|\le C2^{nm+km}  \#(\partial T_k \cap (2^k \Z)^m)\, ,\]
sum in $n$ and  $k$, and use Lemma \ref{lem:bdrytree} to  conclude the desired bound for 
\eqref{eq:diff2}.
This finishes the proof of the lemma.
\end{proof}

 \begin{proof}[Proof of Lemma \ref{lem:tel}]
 We proceed along the lines of the argument in  {\cite[Section 5.2]{DST22}} in the case of the identity matrix, and streamline the proof in our setting.

 Integrating by parts in $q$, we see that 
 \[-2\sum_{i=1}^m\int_{\R^m} (\partial_{i}\partial_{i+m}g)_{Dt}(y+(q,q))\, dq= \int_{\R^m} (\Delta g)_{Dt} (y+(q,q)) \, dq\, .\]
  Using the heat equation 
$(\Delta g)_{tD}
=  2\pi t \partial_t(g_{tD})$,  we thus obtain 
\[ -\alpha^{4m} 4 \pi \sum_{i=1}^m \Theta^{(i)}_{\T,0,\alpha}(F_1,F_2,F_3,F_4)  \]
\[
 =    \sum_{k\in \Z}  \int_{\R^{m}}   \int_{\R^{2m}}   \Big(\prod_{j=1}^4 (F_{j}1_{T_k}) (\rho_j(y)) \Big) \int_{2^{k-1}}^{2^k}
t\partial_{t}(g_{tD})(y+(q,q)) \, \frac{dt}{t} \, 
 dy \, dq\, .\]
 
Let $k_\T$ be defined by $\ell(Q_\T)=2^{k_\mathcal{T}}$. 
 By the fundamental theorem of calculus in $t$, the last display equals
\[ \sum_{k\in \Z} \int_{\R^m}\int_{\R^{2m}}  \Big(\prod_{j=1}^4 (F_{j}1_{T_k}) (\rho_j(y)) \Big)  (g_{2^{k}D} - g_{2^{k-1}D})(y+(q,q)) \, dy \, dq  \]
\begin{equation}\label{e:teles}
= \int_{\R^m}\int_{\R^{2m}}  \Big(\prod_{j=1}^4 (F_{j}1_{Q_{\mathcal{T}}}) (\rho_j(y)) \Big)    g_{2^{k_\T}D}(y+(q,q))\, dy \,  dq   
\end{equation}
\begin{equation}\label{e:teles1}
  + \sum_{k<k_\T} 
\int_{\R^m}\int_{\R^{2m}}   \Big( \prod_{j=1}^4 (F_j1_{T_k}) (\rho_j (y)) 
- \prod_{j=1}^4 (F_j1_{T_{k+1}}) (\rho_j (y))  \Big )
 g_{2^kD}(y+(q,q))\, dy \, dq \, .  
 \end{equation}
We estimate the two terms \eqref{e:teles} and \eqref{e:teles1} separately.

First we estimate \eqref{e:teles1}. Let $\chi$ be the characteristic function of $[-1,1]^{2m}$. 
We bound 
\begin{equation*}
g\le C \sum_{n\ge 0} e^{-2^{n}}\chi_{2^n}\, .
\end{equation*}
We fix $k<k_\T$ and $n\ge 0$, and consider
\begin{equation*}
\int_{\R^m}\int_{\R^{2m}} \Big( \prod_{j=1}^4 (F_j1_{T_k}) (\rho_j (y)) 
- \prod_{j=1}^4 (F_j1_{T_{k+1}}) (\rho_j (y))  \Big )
 \chi_{2^{n+k}D}(y+(q,q)) \, dy \, dq  \, .  
\end{equation*} 
Using the distributive law and $T_k\subseteq T_{k+1}$, we estimate the integrand as
\[\Big| \prod_{j=1}^4 (F_j1_{T_k}) (\rho_j (y)) 
- \prod_{j=1}^4 (F_j1_{T_{k+1}}) (\rho_j (y))  \Big |\]
\[\le \sum_{j_0=1}^4
|F_{j_0}1_{T_{k+1}\setminus T_{k}}| (\rho_{j_0} (y)) 
\prod_{j\neq j_0}
|F_j1_{T_{k+1}}| (\rho_j (y)) \,. \]
 We fix $j_0$. For simplicity of notation we set $j_0=1$, the other values of $j_0$ will be analogous.

Let $Q$ be a cube of side length $2^k$ contained in $T_{k+1}\setminus T_k$ and consider
\begin{equation}
    \label{e:teles3}
\int_{\R^m}\int_{\R^{2m}}   
 |F_11_{Q}| (\rho_{1} (y)) 
\Big( \prod_{j=2}^4
|F_j1_{T_{k+1}}| (\rho_j (y)) \Big )
  \chi_{2^{n+k}D}(y+(q,q)) \, dy \,  dq  \, .  
\end{equation} 
Since $\alpha\geq 1$, we have $y+(q,q)\in 2^{n+k}\alpha[-1,1]^{2m}$. 
Applying $\rho_j$,  we obtain for $1\leq j\leq 4$, 
\[\rho_j(y)+q \in 2^{n+k}\alpha[-1,1]^{m}\, .\]
We also have $\rho_1(y)\in Q$, so $q\in P$, where  
\[P= 2^{n+k}\alpha[-1,1]^{m} -Q \, .\]
Thus, for each $j=2,3,4$, we have $\rho_j(y)\in S$, where 
\[S= Q + 2^{n+k+1}\alpha[-1,1]^{m} \, .\] 
Thus,  we can bound  \eqref{e:teles3} by
\[ 2^{-2m(n+k)}\, (\det D)^{-1}|P|\int_{\R^{2m}}  |F_11_{Q}| (\rho_1 (y)) 
\Big( \prod_{j=2}^4
|F_j1_{T_{k+1}\cap S}| (\rho_j (y)) \Big )\, 
  dy   \, .\]
By  Lemma \ref{lem:bl}, we  estimate this by
\[ C 2^{-2m(n+k)}(\det D)^{-1} |P|
\|F_11_{Q}\|_2\prod_{j=2}^4\|F_j1_{T_{k+1}\cap S}\|_2 \]
\[=  C 2^{-2m(n+k)}(\det D)^{-1} |P| |Q|^{1/2} |S|^{3/2}\Big( \frac{1}{|Q|}\int_{Q} F_1^2\Big)^{1/2} \prod_{j=2}^4  \Big( \frac{1}{|S|}\int_{S} F_j^2 \Big)^{1/2}  . \]
Next, we        crudely estimate $(\det D)^{-1}=\alpha^{-2d_2}\leq 1$ and  $|S|^{3/2}\leq C 2^{2mn}\alpha^{2m}|Q|^{3/2}  $.  We also use that    and $|Q|=C2^{mk}$,  $|P|\leq  C 2^{m(n+k)}\alpha^m$, and that $S$ covers $Q$.  This bounds the last display by
\[C |Q| 2^{mn}\alpha^{3m} \prod_{j=1}^4MF_j(\T)\, .  \]
Summing over the disjoint cubes $Q$ in $T_{k+1}\setminus T_k$, summing over $k<k_{\T}$, and using that the regions $T_{k+1}\setminus T_k$ are disjoint in $Q_\T$,  we then estimate \eqref{e:teles1} by
\[C \Big(  \sum_{n\geq 0} e^{-2^n} 2^{mn} \Big) \alpha^{3m}|Q_{\T}|  \prod_{j=1}^4MF_j(\T) \, . \]
Then  it remains to sum in $n$.

It remains to estimate \eqref{e:teles}, which is done similarly as \eqref{e:teles1} but simpler. 
Estimating the Gaussian by a superposition of characteristic functions of cubes, we   consider
\begin{equation*}
 \int_{\R^m}  \int_{\R^{2m}} 
\Big( \prod_{j=1}^4
(|F_j|1_{Q_{\T}}) (\rho_j (y)) \Big )
  \chi_{2^{n+k}D}(y+(q,q)) \, dy \, dq \, .
\end{equation*} 
This is then  estimated analogously to   \eqref{e:teles3}.
 \end{proof}

 To finish the proof of Proposition  \ref{prop:kernel2} it remains to do a   stopping time argument, similarly as in \cite{twisted, DST22}.  Denote by $\Lambda(f_1,f_2,f_3)$ the form  in the statement of Proposition \ref{prop:kernel2}.
  Let $\mathcal{Q}$ denote the collection of all dyadic cubes in $\R^{4n}$ contained in $[-2^N,2^N]^{4n}$,  with side-lengths in $[2^{-N},2^N]$ for a large $N>0$.   By the monotone convergence theorem, we may assume that in the integral defining $\Lambda$ one has $(t,p)\in \Omega_{\mathcal{Q}}$. 
By homogeneity we may also normalize 
\[\|f_j\|_{p_j} =1\]
for each  $j=1,2,3$. 
Thus, it suffices to prove 
\[|\Lambda(f_1,f_2,f_3)|\leq C (1+|a|)^{16n}\, . \]

For every triple of integers $k=(k_1,k_2,k_3)$, we define 
\[\mathcal{P}_k=\{Q\in \mathcal{Q}: 2^{k_j-1}< \sup_{Q'\supseteq Q } \Big(\frac{1}{|Q'|} \int_{Q'} |f_j|^2 \Big)^{1/2} \leq  2^{k_j}\, \textup{for $j=1,2,3$}\}\, ,\]
where the supremum is over all cubes $Q'$ in $\R^{4n}$ with sides parallel to the coordinate axes.
 Let $\mathcal{P}_k^{\max}$ be the collection of all maximal dyadic cubes in $\mathcal{P}_k$ with respect to  set inclusion. For every $Q\in \mathcal{P}_k^{\max}$, the collection
 \[\mathcal{T}_{Q} = \{Q'\in \mathcal{P}_k: Q'\subseteq Q  \}\]
is a convex tree and for different $Q\in \mathcal{P}_k^{\max}$, the corresponding trees are disjoint. Proposition~\ref{prop:localabc} gives
\[|\Lambda_{\mathcal{T}_Q,a}(f_1,f_2,f_3, 1)| \leq  |Q| \sum_{j=1}^3  \Big(\frac{1}{|Q'|} \int_{Q'} |f_j|^2 \Big)^{1/2} \leq |Q| 2^{k_1+k_2+k_3}\, .\]
Therefore,
\[(1+|a|)^{-16n}|\Lambda (f_1,f_2,f_3)| \leq \sum_{k\in \Z^3} \sum_{Q\in \mathcal{P}_k^{\max}} |\Lambda_{\mathcal{T}_Q,a}(f_1,f_2,f_3,1)|\leq C \sum_{k\in \Z^3} 2^{k_1+k_2+k_3} \sum_{Q\in \mathcal{P}_k^{\max}}  |Q| \, .   \]
By disjointness of the maximal cubes, for each $j=1,2,3$, 
\[\sum_{Q\in \mathcal{P}_k^{\max}}  |Q|  = \Big| \bigcup_{Q\in \mathcal{P}_k^{\max}}  Q  \Big|\subseteq |\{Mf_j>C2^{k_j}\}|\, ,\]
where $Mf_j$ denotes a ``quadratic" variant of the Hardy-Littlewood maximal function
\[Mf_j(x) = \sup_{Q'\ni x} \Big(\frac{1}{|Q'|} \int_{Q'} |f_j(x)|^2 \, dx\Big)^{1/2} \, ,   \]
with supremum is over all cubes $Q'$   with sides parallel to the coordinate axes.
We split $\Z^3=\mathcal{K}_1\cup \mathcal{K}_2\cup\mathcal{K}_3$, where
$\mathcal{K}_j = \{(k_1,k_2,k_3): k_jp_j\geq k_{j'}p_{j'} \textup{ for $j'=1,2,3$}\}. $
Thus, 
\[(1+|a|)^{-16n}|\Lambda (f_1,f_2,f_3)|  \leq \sum_{j=1}^3\sum_{(k_1,k_2,k_3)\in \mathcal{K}_j} 2^{k_1+k_2+k_3}  |\{Mf_j>C2^{k_j}\}| \]
\[= \sum_{j=1}^3 \sum_{k_j\in \Z} 2^{p_jk_j} |\{Mf_j>C2^{k_j}\}|
\prod_{j'\neq j}\sum_{k_{j'}: k_{j'}\leq p_jk_j/p_{j'}} 2^{k_{j'}-\frac{p_jk_j}{p_{j'}}}  \]
\[\leq C \sum_{j=1}^3\|M f_j\|_{p_j}^{p_j} \leq C \sum_{j=1}^3\|f_j\|_{p_j}^{p_j}\leq C \,. \]
This finishes the proof of Proposition \ref{prop:kernel2}.

\appendix
\section{List of indecomposable modules}
\label{sec classification} 

In Tables \ref{table 4sub}--\ref{table Holder} below we list the modules used in the classification Theorems \ref{thm mod class} and \ref{thm Hclass}.
We specify the modules using block matrices
\begin{center}
\begin{tabular}{|>{\centering}p{0.8cm}|>{\centering}p{0.8cm} |>{\centering}p{0.8cm}| p{0.8cm}|}
    \hline
    $A_{10}$ & $A_{11}$     & $A_{12}$ & $A_{13}$\\
    \hline
    $A_{20}$   & $A_{21}$ & $A_{22}$ & $A_{23}$\\
    \hline
\end{tabular}\,.
\end{center}
We define $M$ to be $\R^n$ for some $n$, and identify each subspace $M_i$ with $\R^{n_i}$ as well. The block columns 
\[
    \begin{pmatrix} A_{1i} \\A_{2i} \end{pmatrix}
\]
for $i=0,1,2,3$ then specify the matrices of the embeddings $M_i \to M$ defining the module $\mathbf{M}$, which fixes implicitly also the dimensions of the subspaces $M_i$ and of $M$. Two modules defined like this are isomorphic if the corresponding block matrices can be transformed into each other by row operations on the whole matrix and column operations on each block column.
In terms of the corresponding Brascamp-Lieb data, the transposes of the block columns are the matrices of the maps $\Pi_i$. 

Following the notation of \cite{Med+2004}, we write $I_n$ for the $n \times n$ identity matrix. We denote by $J_n(\lambda)$ an $n \times n$ Jordan block with eigenvalue $\lambda$. An arrow in the superscript of a matrix indicates that a row or column of zeros is to be added in the direction the arrow points, for example $I_n^\uparrow$ is the $(n+1) \times n$ matrix with one row of zeros, followed by the $n \times n$ identity matrix in the rows below.

Finally, the matrix $X= X(P,s)$ in modules $\mathbf{0}$ and $\mathbf{N}_n$ denotes the companion matrix of the polynomial $(P(t))^s$, for some $s \ge 1$ and an irreducible polynomial $P \in \mathbb{R}[t]$ with $P(t) \ne t$ and $P(t) \ne t - 1$.
The companion matrix of a polynomial $Q(t) = t^n + a_{n-1} t^{n-1} + \dotsb + a_0$ is the matrix 
\[
    \begin{pmatrix}
        0  & \dots & 0 & -a_0\\
        1  & \dots & 0 &-a_1\\
        \vdots & \ddots & \vdots & \vdots\\
        0  & \dots & 1 & -a_{n-1}
    \end{pmatrix}\,,
\]
with characteristic polynomial $Q$.
Note that the conditions on $P$ imply that $P(t) = t - \lambda$ with $\lambda \neq 0, 1$ or $P(t) = t^2 - 2\lambda t + \mu$ with $\mu > \lambda^2$. 

The indecomposable modules in Table \ref{table 4sub} are only listed up to permutation of the subspaces. The additional information which permutations give rise to non-
isomorphic modules is given by the following lemma, which is Remark 1 in \cite{Med+2004}.
\begin{lemma}[{\cite[Remark 1]{Med+2004}}]
    \label{lem perm}
     For the modules $\mathbf{I\!I}, \mathbf{I\!I\!I}, \mathbf{I\!I\!I^*}, \mathbf{IV}, \mathbf{IV^*}, \mathbf{V}, \mathbf{V}^*$, each permutation of the subspaces that leaves their dimensions invariant gives rise to an isomorphic module. For the modules of type $\mathbf{0}$, all permutations of the subspaces give rise to another module of type $\mathbf{0}$, but possibly with different $X$. For module $\mathbf{I}$, swapping the columns $1,3$ or swapping columns $2,4$ gives rise to an isomorphic module. Thus there are $6$ isomorphism classes of modules that can be obtained by permuting the subspaces in type $\mathbf{I}$. 
\end{lemma}

Lemma \ref{lem perm} can be directly checked by transforming the corresponding block matrices into each other using the allowed row and column transformations. We omit this and refer to \cite{Med+2004}. 

\begin{table}
    \centering
\begin{tabular}{c c c}
$\mathbf{M}$ & $\dim (M, M_0, M_1, M_2, M_3)$ & block matrix\\
\hline\hline
$\mathbf{0}$& $(2n, n,n,n,n)$ &
\begin{tabular}{|>{\centering}p{0.8cm}|>{\centering}p{0.8cm} |>{\centering}p{0.8cm}| >{\centering\arraybackslash} p{0.8cm}|}
    \hline
    $I_{n}$ & $0$     & $I_n$ & $X$\\
    \hline
    $0$   & $I_n$ & $I_n$ & $I_n$\\
    \hline
\end{tabular}\\
\hline
$\mathbf{I}$& $(2n, n,n,n,n)$ &
\begin{tabular}{|>{\centering}p{0.8cm}|>{\centering}p{0.8cm} |>{\centering}p{0.8cm}|  >{\centering\arraybackslash} p{0.8cm}|}
    \hline
    $I_{n}$ & $0$     & $I_n$ & $J_n(0)$\\
    \hline
    $0$   & $I_n$ & $I_n$ & $I_n$\\
    \hline
\end{tabular}\\
\hline
$\mathbf{I\!I}$& $(2n+1, n+1,n+1,n,n)$ &
\begin{tabular}{|>{\centering}p{0.8cm}|>{\centering}p{0.8cm} |>{\centering}p{0.8cm}| >{\centering\arraybackslash} p{0.8cm}|}
    \hline
    $I_{n+1}$ & $I_{n+1}$     & $I_n^\downarrow$ & $0$\\
    \hline
    $0$   & $I_n^\rightarrow$ & $I_n$ & $I_n$\\
    \hline
\end{tabular}\\
\hline
$\mathbf{I\!I\!I}$& $(2n+1, n,n,n,n+1)$ &
\begin{tabular}{|>{\centering}p{0.8cm}|>{\centering}p{0.8cm} |>{\centering}p{0.8cm}|  >{\centering\arraybackslash}p{0.8cm}|}
    \hline
    $I_{n+1}$ & $0$     & $I_n^\uparrow$ & $I_n^\downarrow$\\
    \hline
    $0$   & $I_n$ & $I_n$ & $I_n$\\
    \hline
\end{tabular}\\
\hline
$\mathbf{I\!I\!I}^*$& $(2n+1, n+1,n+1,n+1,n)$ &
\begin{tabular}{|>{\centering}p{0.8cm}|>{\centering}p{0.8cm} |>{\centering}p{0.8cm}| >{\centering\arraybackslash} p{0.8cm}|}
    \hline
    $I_{n}$ & $0$     & $I_n^\leftarrow$ & $I_n^\rightarrow$\\
    \hline
    $0$   & $I_{n+1}$ & $I_{n+1}$ & $I_{n+1}$\\
    \hline
\end{tabular}\\
\hline
$\mathbf{IV}$ & $(2n+2, n+1,n+1,n+1,n)$ &
\begin{tabular}{|>{\centering}p{0.8cm}|>{\centering}p{0.8cm} |>{\centering}p{0.8cm}| >{\centering\arraybackslash} p{0.8cm}|}
    \hline
    $I_{n+1}$ & 0     & $I_{n+1}$ & $I_n^\uparrow$\\
    \hline
    $0$   & $I_{n+1}$ & $I_{n+1}$ & $I_n^\downarrow$\\
    \hline
\end{tabular}\\
\hline
$\mathbf{IV}^*$ & $(2n+2, n+1,n+1,n+1,n+2)$ &
\begin{tabular}{|>{\centering}p{0.8cm}|>{\centering}p{0.8cm} |>{\centering}p{0.8cm}| >{\centering\arraybackslash} p{0.8cm}|}
    \hline
    $I_{n+1}$ & 0     & $I_{n+1}$ & $I_{n+1}^\leftarrow$\\
    \hline
    $0$   & $I_{n+1}$ & $I_{n+1}$ & $I_{n+1}^\rightarrow$\\
    \hline
\end{tabular}\\
\hline
$\mathbf{V}$ & $(2n+1, n,n,n,n)$ &
\begin{tabular}{|>{\centering}p{0.8cm}|>{\centering}p{0.8cm} |>{\centering}p{0.8cm}| >{\centering\arraybackslash} p{0.8cm}|}
    \hline
    $I_n$ & 0     & $J_n(0)$ & $I_n$\\
    \hline
    $0$   & $I_n$ & $I_n$ & $J_n(0)$\\
    \hline
    $0..0$ & $0..0$ & $10..0$ & $10..0$\\
    \hline
\end{tabular}\\
\hline
$\mathbf{V}^*$ & $(2n+1,n+1,n+1,n+1,n+1)$ &
\begin{tabular}{|>{\centering}p{0.8cm}|>{\centering}p{0.8cm} |>{\centering}p{0.8cm}|  >{\centering\arraybackslash} p{0.8cm}|}
    \hline
    $I_{n}^\leftarrow$ & $I_{n}^\leftarrow$     & $I_n^{\rightarrow}$ & $0$\\
    \hline
    $0$   & $I_n^\rightarrow$ & $I_{n}^\leftarrow$ & $I_{n}^\leftarrow$\\
    \hline
    $10..0$ & $10..0$ & $10..0$ & $10..0$\\
    \hline
\end{tabular}\\ 
\hline
\hline
\vspace{0.5em}
\end{tabular}
\caption{Indecomposable modules of the four subspace quiver, up to permutation of the subspaces. The following list is a direct result from the diagrams in \cite{Med+2004}.}
\label{table 4sub}
\end{table}

\begin{table}
    \centering
\begin{tabular}{c c c}
$\mathbf{M}$ & $\dim (M, M_0, M_1, M_2, M_3)$ & block matrix\\
\hline\hline
$\mathbf{N}_n$& $(2n, n,n,n,n)$ &
\begin{tabular}{{|>{\centering}p{0.8cm}|>{\centering}p{0.8cm} |>{\centering}p{0.8cm}| >{\centering\arraybackslash} p{0.8cm}|}}
    \hline
    $I_n$ & 0     & $I_n$ & $X$\\
    \hline
    $0$   & $I_n$ & $I_n$ & $I_n$\\
    \hline
\end{tabular}\\
\hline
$\mathbf{J}^{(1)}_n$ & $(2n, n,n,n,n)$ &
\begin{tabular}{{|>{\centering}p{0.8cm}|>{\centering}p{0.8cm} |>{\centering}p{0.8cm}| >{\centering\arraybackslash} p{0.8cm}|}}
    \hline
    $I_n$ & 0     & $I_n$ & $J_n(1)$\\
    \hline
    $0$   & $I_n$ & $I_n$ & $I_n$\\
    \hline
\end{tabular}\\
\hline
$\mathbf{J}^{(2)}_n$ & $(2n, n,n,n,n)$ &
\begin{tabular}{{|>{\centering}p{0.8cm}|>{\centering}p{0.8cm} |>{\centering}p{0.8cm}| >{\centering\arraybackslash} p{0.8cm}|}}
    \hline
    $I_n$ & 0     & $I_n$ & $J_n(0)$\\
    \hline
    $0$   & $I_n$ & $I_n$ & $I_n$\\
    \hline
\end{tabular}\\
\hline
$\mathbf{J}^{(3)}_n$ & $(2n, n,n,n,n)$ &
\begin{tabular}{{|>{\centering}p{0.8cm}|>{\centering}p{0.8cm} |>{\centering}p{0.8cm}| >{\centering\arraybackslash} p{0.8cm}|}}
    \hline
    $I_n$ & 0     & $J_n(0)$ & $I_n$\\
    \hline
    $0$   & $I_n$ & $I_n$ & $I_n$\\
    \hline
\end{tabular}\\
\hline
$\mathbf{C}_n$ & $(2n+1,n+1,n,n,n)$ &
\begin{tabular}{{|>{\centering}p{0.8cm}|>{\centering}p{0.8cm} |>{\centering}p{0.8cm}| >{\centering\arraybackslash} p{0.8cm}|}}
    \hline
    $I_{n+1}$ & 0     & $I_n^{\uparrow}$ & $I_n^{\downarrow}$\\
    \hline
    $0$   & $I_n$ & $I_n$ & $I_n$\\
    \hline
\end{tabular}\\
\hline
$\mathbf{T}_n$ & $(2n+1,n, n+1, n+1, n+1)$ &
\begin{tabular}{{|>{\centering}p{0.8cm}|>{\centering}p{0.8cm} |>{\centering}p{0.8cm}| >{\centering\arraybackslash} p{0.8cm}|}}
    \hline
    $I_n$ & 0     & $I_n^{\leftarrow}$ & $I_n^{\rightarrow}$\\
    \hline
    $0$   & $I_{n+1}$ & $I_{n+1}$ & $I_{n+1}$\\
    \hline
\end{tabular}\\
\hline
\hline
\vspace{0.5em}
\end{tabular}
\caption{Indecomposable modules corresponding to data of Hölder type}\label{table htype}
\end{table}

\begin{table}
    \centering
\begin{tabular}{c c c}
$\mathbf{M}$ & $\dim (M, M_0, M_1, M_2, M_3)$ & block matrix\\
\hline\hline
$\mathbf{Y}$ & $(2, 0, 1, 1, 1)$ &
\begin{tabular}{{|>{\centering}p{0.8cm}|>{\centering}p{0.8cm} |>{\centering}p{0.8cm}| >{\centering\arraybackslash} p{0.8cm}|}}
    \hline
    $0$ & $1$ & $0$     & $1$\\
    \hline
    $0$ & $0$   & $1$ & $1$\\
    \hline
\end{tabular}\\
\hline
$\mathbf{Z}$ & $(3, 1,1,1,1)$ &
\begin{tabular}{{|>{\centering}p{0.8cm}|>{\centering}p{0.8cm} |>{\centering}p{0.8cm}| >{\centering\arraybackslash} p{0.8cm}|}}
    \hline
    $1$ & $0$     & $0$ & $1$\\
    \hline
    $0$   & $1$ & $1$ & $0$\\
    \hline
    $0$ & $0$ & $1$ & $1$\\
    \hline
\end{tabular}\\
\hline
\hline
$\mathbf{L}$ & $(4, 1, 2, 2, 2)$ &
\begin{tabular}{{|>{\centering}p{0.8cm}|>{\centering}p{0.8cm} |>{\centering}p{0.8cm}| >{\centering\arraybackslash} p{0.8cm}|}}
    \hline
    $I_1^{\uparrow}$ & $I_2$ & 0     & $I_2$\\
    \hline
    $I_{1}^\downarrow$ & $0$   & $I_{2}$ & $I_{2}$\\
    \hline
\end{tabular}\\
\hline
$\mathbf{B}$ & $(5, 2,2,2,2)$ &
\begin{tabular}{{|>{\centering}p{0.8cm}|>{\centering}p{0.8cm} |>{\centering}p{0.8cm}| >{\centering\arraybackslash} p{0.8cm}|}}
    \hline
    $I_2$ & 0     & $J_2(0)$ & $I_2$\\
    \hline
    $0$   & $I_2$ & $I_2$ & $J_2(0)$\\
    \hline
    $00$ & $00$ & $10$ & $10$\\
    \hline
\end{tabular}\\
\hline
\hline
\vspace{0.5em}
\end{tabular}
\caption{Indecomposable modules corresponding to Young's convolution inequality and to Loomis-Whitney type inequalities}
%\label{table YLM}
\end{table}

\begin{table}
    \centering
\begin{tabular}{c c c}
$\mathbf{M}$ & $\dim (M, M_0, M_1, M_2, M_3)$ & block matrix\\
\hline\hline
$\mathbf{P}^{(1)}$ & $(1, 0, 0, 1, 1)$ &
\begin{tabular}{{|>{\centering}p{0.8cm}|>{\centering}p{0.8cm} |>{\centering}p{0.8cm}| >{\centering\arraybackslash} p{0.8cm}|}}
    \hline
    $0$ &  $0$  & $I_1$ & $I_1$\\
    \hline
\end{tabular}\\
\hline
$\mathbf{P}^{(2)}$ & $(1, 0, 1, 0, 1)$ &
\begin{tabular}{{|>{\centering}p{0.8cm}|>{\centering}p{0.8cm} |>{\centering}p{0.8cm}| >{\centering\arraybackslash} p{0.8cm}|}}
    \hline
    $0$ &  $I_1$  & $0$ & $I_1$\\
    \hline
\end{tabular}\\
\hline
$\mathbf{P}^{(3)}$ & $(1, 0, 1, 1, 0)$ &
\begin{tabular}{{|>{\centering}p{0.8cm}|>{\centering}p{0.8cm} |>{\centering}p{0.8cm}| >{\centering\arraybackslash} p{0.8cm}|}}
    \hline
    $0$ &  $I_1$  & $I_1$ & $0$\\
    \hline
\end{tabular}\\
\hline
$\mathbf{K}^{(1)}$ & $(2, 1,0,1,1)$ &
\begin{tabular}{{|>{\centering}p{0.8cm}|>{\centering}p{0.8cm} |>{\centering}p{0.8cm}| >{\centering\arraybackslash} p{0.8cm}|}}
    \hline
    $I_1$ & $0$ & $0$ & $I_1$\\
    \hline
    $0$ & $0$ & $I_1$ & $I_1$\\
    \hline
\end{tabular}\\
\hline
$\mathbf{K}^{(2)}$ & $(2, 1,1,0,1)$ &
\begin{tabular}{{|>{\centering}p{0.8cm}|>{\centering}p{0.8cm} |>{\centering}p{0.8cm}| >{\centering\arraybackslash} p{0.8cm}|}}
    \hline
    $I_1$ & $0$ & $0$ & $I_1$\\
    \hline
    $0$ & $I_1$ & $0$ & $I_1$\\
    \hline
\end{tabular}\\
\hline
$\mathbf{K}^{(3)}$ & $(2, 1,1,1,0)$ &
\begin{tabular}{{|>{\centering}p{0.8cm}|>{\centering}p{0.8cm} |>{\centering}p{0.8cm}| >{\centering\arraybackslash} p{0.8cm}|}}
    \hline
    $I_1$  & $0$ & $I_1$ & $0$\\
    \hline
    $0$  & $I_1$ & $I_1$ & $0$\\
    \hline
\end{tabular}\\
\hline
\hline
\vspace{0.5em}
\end{tabular}
\caption{Indecomposable modules corresponding to Hölder's inequality or Hölder's inequality combined with boundedness of a linear singular integral operator}
\label{table Holder}
\end{table}

\newpage
\bibliographystyle{abbrv}
\bibliography{ref}

\end{document}